\documentclass[11pt]{amsart}
\usepackage{saltscribe}
\usepackage[shortlabels]{enumitem}

\newcommand{\bd}{\mathbf{d}}
\newcommand{\oR}{\bR^4_\leq}

\newcommand{\ce}{\mathrm{Cone}_e}
\newcommand{\bg}{\bar{\gamma}}

\newcommand{\li}{\mathrm{Li}}
\newcommand{\ls}{\mathrm{Ls}}

\begin{document}
\title{Large Deviation Principle for Last Passage Percolation Models}
\author{Pranay Agarwal}
\address{Pranay Agarwal, Department of Mathematics, University of Toronto, Toronto, Canada.}
\email{pranay.agarwal@mail.utoronto.ca}

\begin{abstract}
	Study of the KPZ universality class has seen the emergence of universal objects over the past decade which arise as the scaling limit of the member models. One such object is the directed landscape, and it is known that exactly solvable last passage percolation (LPP) models converge to the directed landscape under the KPZ scaling (see \cite{DV21}). Large deviations of the directed landscape on the metric level were recently studied in \cite{DDV24}, which also provides a general framework for establishing such large deviation principle (LDP). The main goal of the article is to apply and refine that framework to establish a LDP for LPP models at the metric level without relying on exact solvability. We then use the LDP on the metric level to establish a LDP for geodesics in these models, thus providing a streamlined way to study large transversal fluctuations of geodesics in these models. We briefly touch on how the theory extends to other planar models like directed polymers and Poisson LPP.
\end{abstract}

\maketitle

\setcounter{tocdepth}{1}
\tableofcontents

\section{Introduction} \label{s: intro}
We consider last passage percolation (LPP) models on the discrete lattice $\bZ^2$ with independent and identically distributed (i.i.d.)\,non-negative weights. In these models, we assign i.i.d.\,random weights $(w_{ij} \sim \mu)$ to each vertex $(i,j)$ of $\bZ^2$. Weight/length of a path in $\bZ^2$ is the sum of the weight of the vertices it passes through (both endpoints included). Given a pair of ordered vertices $u,v$ in $\bZ^2$, we are interested in the maximal length of an upright path between the two vertices. This maximal length is referred to as last passage time $T\br{u,v}$ and the paths which achieve this maximal weight are called geodesics. We extend last passage times to points $(u,v) =(x_1,y_1,x_2,y_2) \in \bR^4$ as follows
\[
	T(u, v) := \begin{cases}
		T((\ceil{x_1}, \ceil{y_1}), (\floor{x_2}, \floor{y_2})) & \mathrm{if} \; \ceil{x_1} \leq \floor{x_2} \;\mathrm{and} \; \ceil{y_1} \leq \floor{y_2},\\
		0 & \mathrm{otherwise.}
	\end{cases}
\] 
It is often easier to work with such models after a change of coordinate axes. Hence, we shall identify points $(a,b) \in \bR^2$ with their space time coordinates $(x,t) := (b-a,b+a)$. Thus, the first quadrant gets mapped to $$\vee :=  \bc{(x,t) \in \bR \times [0, \infty) : \abs{x} \leq t}.$$ For $p \in \bR^2$, we define $\vee_p$ as the translation of $\vee$ by $p$. Thus, $\vee_p = \bc{p + u : u \in \vee}$.

Under this change of coordinates, last passage times define upper semi-continuous functions on $\oR := \bc{(p;q) \in \bR^4 \:\vert\: q \in \vee_p}$. As LPP models assign random lengths to paths, one would expect them to satisfy some form of the triangle inequality. However, last passage times cannot be defined in a way so that they satisfy both upper semi-continuity and a triangle inequality. Let $\cE$ denote the set of all non-negative upper semi-continuous functions $e$ on $\oR$ which satisfy a generalized reverse triangle inequality, 
\begin{equation} \label{eq: triangle}
	e(p;q) \geq e(p;r) +e(r^+, q),
\end{equation}
for $(p;q), (p;r), (r;q) \in \oR$, where 
\begin{equation} \label{def: plus}	
e(r^+, q) := \begin{cases}
	\sup \bc{e(r'; q) \:\vert\: r' \in \vee_r \setminus \bc{r}}  \; &\mathrm{if} \;r \neq q,\\
	0 \;\;&\mathrm{otherwise}.
\end{cases}
\end{equation}  
Generalizing the triangle inequality in this way allows last passage times to retain the property while preserving upper semi-continuity. Hence, last passage times can be viewed as random instances of elements of $\cE$. Let $u =(p ,q) \in \oR$, we shall focus on the rescaled passage times
\[T_n(u) := T(np,nq)/n .\]
Let us now discuss some characteristic properties of LPP models before we discuss this further.
 
We begin by introducing the shape function, the study of which has been a fundamental problem in LPP models. The shape function $F :\vee \to [0, \infty]$ is defined as
\begin{equation} \label{def: shape}
	F(v) = \sup_n {\E{T_n((0,0), v)}}, \;\text{for }v \in \vee.
\end{equation}
One can prove the finiteness of the shape function by leveraging the superadditivity of last passage times and Kingman's subadditivity theorem, under sufficient moment assumptions. Universal properties of $F$ were established in \cite{Mart04}. In particular, it was shown that shape function is concave and continuous over $\vee$. This is the first property of our models that will be relevant to our discussion.

\begin{property}
	\label{A1}
	The shape function $F$ as defined in \eqref{def: shape} is continuous, concave and sublinear over $\vee$. Furthermore, the supremum can be realized as a limit. For $v \in \vee$, 
	\[
		\lim_{n \to \infty} {T_n\br{(0,0), v}} = F(v),
	\]
	uniformly on compact sets, almost surely.
\end{property}

Deriving the explicit form of the shape function requires a much finer analysis. Variational formulas that characterize the limiting free energies and limit shape for a general class of polymer models were established in \cite{GRS16}. The shape function $F$ is explicitly known for a select few such models, such as exponential and geometric LPP (see \cite{Ro81}).

Large deviations of the limit law in Property \ref{A1} will be central to our arguments. Under sufficient moment conditions, it can be shown that the limit law satisfies at least two large deviation principles, one being at the speed of $n^{-1}$, which is referred to as the upper tail deviations. The existence and convexity of the upper tail large deviation rate function $J$ follows from the underlying superadditivity in the model. We shall require that $J$ also satisfies some basic regularity properties. 

\begin{property}
	\label{A2}
	For $\abs{\gamma} \leq 1$, there exists function $J:[-1,1] \times \bR \to [0, \infty]$ such that
	\[
		\lim_{n \to \infty} \frac{1}{n} \log \P{T_n \br{(0,0), (\gamma , 1)} \geq x} = -J(\gamma, x).
	\] 
	Furthermore, the function $J$ can be computed as
	\begin{equation}
		-J(\gamma, x) = \sup_n \frac{1}{n} \log \P{T_n \br{(0,0), ({\gamma }, 1)} \geq x},
	\end{equation}
	and $J$ satisfies the following properties:
	\begin{enumerate}[(i)]
		\item The function $J$ is jointly convex and continuous. 
		\item The function $J_\gamma = J(\gamma, \cdot)$ is zero for $x \leq F(\gamma, 1)$ and strictly increasing for $x \geq F(\gamma, 1)$.
	\end{enumerate}
\end{property}
Explicit form of $J$ is known for geometric LPP, as established in \cite{Jo00}. The rate function for exponential LPP can be established using the same arguments or by viewing exponential weights as a limit of geometric weights. In both cases, the rate function satisfies Property \ref{A2} and this observation is true for a wider class of LPP models, even when the explicit form of the rate function is not known (detailed discussion in Section \ref{sec: properties}).

The lower tail of $T$ decays at a much faster rate, and we shall refer to these events as lower tail large deviations. The lower tail large deviations are more complicated to study as one can no longer rely on the superadditivity of the model. However, for our arguments we only require some quantitative decay for the probability of the lower tail large deviation event.
\begin{property}
	\label{A3}
	For $\abs{\gamma} < 1$,
	\[
		\P{T_n((0,0), (\gamma , 1)) \leq F(\gamma, 1) - x} \leq \exp \bc{-\omega(n)},
	\]
	where $\omega(n)$ is used to denote a function which grows faster than any linear function of $n$.
\end{property}
For exactly solvable LPP models, the above probability is known to decay exponentially in $n^2$. Techniques to derive the rate function at this speed are discussed in \cite{Jo00}. Rate functions have also been studied for the closely related model of longest increasing subsequences in random permutation in \cite{DZ99} which prove the lower tail rate function while the upper tail rate function was established in \cite{seppalainen1998large}. 

Note that all the above properties extend to last passage times across arbitrary endpoints by the spatial stationarity of LPP models. We shall show in Section \ref{sec: properties}, that all the aforementioned properties hold when the i.i.d.\,random weights follow a distribution $\mu$ supported on the non-negative reals with a well-defined moment generating function in a neighborhood of zero, i.e.,
\begin{equation} \label{ass: mom}
	\exists t >0 \text{ such that } \E[\mu]{e^{t x}} < \infty.
\end{equation}
We shall additionally require that the support of $\mu$ is unbounded. This is required for the continuity claim in Property \ref{A2}(i), which is crucial for the lower semi-continuity of the metric rate function. 

By Property \ref{A1}, the rescaled passage times almost surely converge uniformly on compact sets to the function
\begin{equation}
	d(p,q) :=	F(q-p) \;\; \mathrm{for} \; (p,q) \in \bR^4_<\:.
\end{equation}
Since, we will be treating $T_n$ as random elements of $\cE$, a more natural mode of convergence for us will be the hypo-convergence of upper semi-continuous functions. We note that uniform convergence on compact sets is stronger than hypo-convergence, hence we have almost sure hypo-convergence of $T_n$ to $d$ in light of Property \ref{A1}. The goal of this article is to establish a large deviation principle for this almost sure convergence.

\begin{theorem}
	\label{thm: ldp}
	There exists a lower semi-continuous function $I : \cE \to [0, \infty]$, such that for every Borel measurable set $\cA \subseteq \cE$, as $n \to \infty$ we have
	\[
		\exp\br{(o(1)- \inf_{A^\circ} I) n} \leq \P{T_n \in \cA} \leq \exp\br{(o(1) -\inf_{\bar{\cA}} I) n}.
	\]
	Furthermore, $I$ is a good rate function i.e., $I^{-1}([0,a])$ is compact for all $a < \infty$. We also show that $I^{-1}(0) =d$ and $I$ is strictly increasing on $I^{-1}([0, \infty))$.
\end{theorem}
By Property \ref{A3}, the lower tail deviation events will have infinite rate at this speed while upper tail large deviation will be captured by this LDP. Hence, we shall refer to this LDP as the upper tail LDP. We expect another rate function at the speed of $n^{-2}$ corresponding to lower tail events. Two separate rate functions at different speeds were proved recently for the closely related first passage percolation (FPP) model. An upper tail large deviation principle at the metric level for FPP was established in \cite{Ver24a} at the speed of $n^{-2}$ (at a speed of $n^{-d}$ in $\bZ^d$). This was soon followed by a lower tail large deviation principle in \cite{Ver24b} at the speed of $n^{-1}$. Note that upper/lower tail events in FPP are related to lower/upper tail events in LPP, hence the swap in speeds. The latter result was proven for bounded random weights by interpreting passage times as random bounded functions and derives an identical rate function as ours. Just like how our theory cannot be naturally extended to bounded variables, the LDP in \cite{Ver24b} cannot be extended to the unbounded case. Thus, our theories complement each other and completely answers questions regarding upper tail large deviations at the metric level.

Using the metric level large deviation principle, we shall study the large deviations of the geodesic between a given set of endpoints. For the sake of simplicity, we shall assume for this part that $\mu$ defines continuous random variables. This ensures that the last passage geodesics are almost surely unique between any two fixed endpoints. Let $u =(x,s;y,t) \in \oR$ and $\gamma_n$ denote the a.s.\,unique geodesic of $T_n$ with endpoints $u$. We shall treat $\gamma_n$ as elements of the space of continuous functions $C([s,t])$ with the topology of uniform convergence. 
\begin{theorem} \label{thm: geodesics}
	For every Borel measurable set $A \subset C([s,t])$ we have
	\[
		\exp \br{(o(1) - \inf_{A^\circ} \cJ_u) n} \leq \P{\gamma_{n} \in A} \leq \exp \br{(o(1) - \inf_{\bar{A}} \cJ_u) n}.
	\]
	Here $\cJ_u : C([s,t]) \to [0, \infty]$ is a good rate function satisfying
	\begin{align} \label{def: Ju}
		\cJ_u(f) :=\begin{cases}
			\inf \bc{I(e) : e \in \cD_u(f)}, \;\;\;\;\text{when}\;\;\;\;\; \cD_u(f) \neq \emptyset,\\
			\infty, \;\;\;\;\;\mathrm{otheriwse.}
		\end{cases}
	\end{align}
	where $\cD_u(f) \subset \cE$ is the set of metrics such that $f$ is a geodesic with endpoints $u$. 
\end{theorem}

Large deviations of the geodesics at the middle of the journey were recently studied in \cite{ABGS25}. We shall provide a short proof of their result as an application of Theorem \ref{thm: geodesics}. 
\begin{corollary} \label{cor: corner}
	Let $u =(0,0;0,1) \in \oR$ and fix $t \in (0, 1/2)$. Let $\gamma_n(s)$ denote the $x$ coordinate of geodesic $\gamma_n$ at time $s$, 
	\begin{equation} \label{eq: corn-prob}
		\P{\gamma_n(1/2) \geq t} = \exp\br{(o(1) - J(2t, d(u)))n}.
	\end{equation}
\end{corollary}
We note that our techniques only allow us to establish the upper bound for the case of $t=1/2$. The matching lower bound would require a constructive proof which is detailed in \cite{ABGS25}.

\subsection*{Organization of Paper} We discuss some necessary topological considerations in Section \ref{sec: pre}. This section is mainly an overview of the topology of hypo-convergence. We then discuss some intuition behind the rate function and describe a preliminary rate function in Section \ref{sec: prf}. The preliminary rate function allows us to identify the correct candidates for finite rate metrics. Section \ref{sec: frm} deals with the topological properties of the set of finite rate metrics and establishes that these metrics define geodesic spaces. This property is then used extensively to understand the finer structure of finite rate metrics and realize them as planted network metrics in Section \ref{sec: ntwrk}. We also state and prove the claimed properties of the actual rate function in this section. The next two sections are dedicated to proving the large deviation upper bound and lower bound. The upper bound can be proven with only the theory developed in \ref{sec: frm}. The only input that we need from Section \ref{sec: ntwrk} is the definition of the true rate function. The lower bound is not so simple and requires an understanding of the finite rate metrics as planted network metrics to prove the desired results. Section \ref{sec: geo} then builds upon Theorem \ref{thm: ldp} to prove a large deviation principle for geodesics in models where the geodesic is almost surely unique. Finally, we see the proof of the key properties (Property \ref{A1}, \ref{A2}, \ref{A3}) to complete the picture for last passage percolation on lattice models. We discuss implications for directed polymers and Poisson LPP in Section \ref{sec: pois}. 

\section{Preliminaries} \label{sec: pre}

\subsection*{Hypo-convergence.}
Let $X$ be a metrizable topological space and $UC(X)$ denote the collection of all real-valued upper semi-continuous functions. For an upper semi-continuous function $f: X \to \bR$, we define the hypograph $hypo(f)$ as
\[hypo(f) := \bc{(x,t) : x \in X, t \leq f(x)}.\]
Hypographs of upper semi-continuous functions are closed subsets of the product space $X \times \bR$. Thus, $hypo(f)$ provides an identification of $f$ as an element of the hyperspace of $X \times \bR$, which is the collection of all closed sets in $X \times \bR$. One can talk about convergence of functions in terms of the Painlev\'e-Kuratowski convergence of their hypographs. Given closed subsets $A_1, A_2, \cdots$ of a topological space $Y$, the sequence $(A_n)$ is called Painlev\'e-Kuratowski convergent to $A$ provided that $\li A_n =\ls A_n =A$, where
\begin{align*}
	\li A_n &:= \bc{a \in Y : \exists (a_n) \in Y,  a_n \to a, a_n \in A_n \text{ for all but finitely many }n},\\
	\ls A_n &:= \bc{a \in Y : \exists (n_k) \subset \bN, a_k \in A_{n_k}, \text{such that }  a_k \to a}.
\end{align*}
We say that $(f_n) \in UC(X)$ is hypo-convergent to $f$, provided that their hypographs are Painlev\'e-Kuratowski convergent to $hypo(f)$. This is equivalent to the following two conditions:
\begin{enumerate}[H(I)]
	\item: Whenever $x_n \to x$, then $\limsup_{n}f_n(x_n) \leq f(x)$. \label{hypo1}
	\item: There exists a sequence $\bc{x_n}$ convergent to $x$ such that $\lim_n f_n(x_n) = f(x)$. \label{hypo2}
\end{enumerate}
Painlev\'e-Kuratowski convergence is equivalent to the Fell topology, thus induces a topology, which we denote by $\tau_h$. This topology is generated by the following subbase of open sets:
\begin{align*}
	\bc{\cB^{K,a}, \cB_{G,a} : K, G \subset X, K\text{ compact}, G\text{ open}, a \in \bR \cup \bc{\infty}},
\end{align*}
where for any $Q \subset X$ and $a \in \bR \cup \bc{\infty}$, 
\begin{align*}
	\cB^{Q,a} =\bc{f \in UC(X) : \sup_Q f(x) < a} && \mathrm{and}&& 
	\cB_{Q,a} =\bc{f \in UC(X) : \sup_Q f(x) > a}.
\end{align*}
As a consequence of this equivalence, we have that $\tau_h$ is Hausdorff iff $X$ is locally compact (see \cite[Theorem 2.76]{Att84}). More interestingly, we have that
\begin{theorem}[{\cite[Corollary 2.79]{Att84}, \cite[Corollary 4.3]{DSW83}}] \label{thm: AA}
	Let $X$ be a locally compact Hausdorff second countable space, then the topological space $(UC(X), \tau_h)$ is compact Hausdorff and second countable.
\end{theorem}

An immediate consequence of the above statement along with Urysohn's metrization theorem, is that $(UC(X), \tau_h)$ is metrizable. Thus, we can choose a metric $\bd$ on $UC(X)$ which induces the topology $\tau_h$.

Convergence of upper semi-continuous functions can also be described in terms of the convergence of their hypographs in the Hausdorff metric. The Hausdorff distance $d_h$ between two closed sets $C$ and $K$ of $X \times \bR$ is defined as 
\[d_h(C,K) := \inf \bc{\lambda > 0 : K \subset B_\lambda(C), C \subset B_\lambda(K) } ,\]
where $B_\lambda(K)$ is the union of all $\lambda$ balls with center in $K$. One can now define the Hausdorff distance between upper semi-continuous functions as the Hausdorff distance of their hypographs. For $f \in UC(X), \lambda >0$ we define $f^\lambda: X \to \bR$ by
\[f^\lambda(x) := \sup \bc{\alpha : (x, \alpha) \in B_\lambda(hypo(f))}. \] 
Then, $f^\lambda \in UC(X)$ and $d_h(f^\lambda, f)=\lambda$. Moreover, $d_h(f,g) \leq \lambda$ iff $f \leq g^\lambda$ and $g \leq f^\lambda$. This observation yields the following equivalence:

\begin{remark} \label{rmk: hc}
	If $X$ is a compact metric space then one can show that \ref{hypo1} is equivalent to $f_n \leq f^\lambda$ for any $\lambda >0$ for $n$ large enough. Similarly, \ref{hypo2} is equivalent to $f \leq f^\lambda_n$ for any $\lambda >0$ for $n$ large enough. Thus, if $X$ is a compact metric space then the hypo-convergence of $(f_n) \in UC(X)$ is also equivalent to the convergence of the hypographs in the Hausdorff metric. 	
\end{remark}

We refer the readers to \cite{BEER19821} for more details on the equivalence. In our setting, the space $X$ shall be $\oR$. We endow $\oR$ with the norm $d^1((x,s;y,t)) := \max(\abs{x} + \abs{s}, \abs{y} + \abs{t})$ as it behaves well with the structure of last passage percolation.

\subsection*{Notation}
For points $p,q \in \bR^2$, we say $p \leq q$ if $(p,q) \in \bR^4_<$. We shall drop the equality sign to imply that $p \leq q$ and $p \neq q$. We shall also make use of some asymptotic notations. For functions $f, g: \bR \to [0, \infty]$, we say $f = O(g)$ if there exists a constant $c$ such that $f \geq cg$. Similarly, we say $f =\omega(g)$ to imply that for all $c >0$, $f(x) \geq cg(x)$ for $x$ sufficiently large.

\section{Preliminary Rate Function} \label{sec: prf}
In this section, we introduce a preliminary rate function $\Theta$. It coincides with the actual rate function over the subset of $\cE$ containing functions with finite rate and helps develop the machinery required to define the true rate function. We begin by discussing the heuristics behind the choice of $\Theta$.


Given $e \geq d,$ we want to understand the probability that $T_n$ is close to $e$. Owing to the similarity in the problems, our strategy to develop the rate function will be similar to that adopted in \cite{DDV24}. Let $I$ be an interval, we call a continuous curve $\gamma: I \to \bR$ a path in $\oR$ if $(\gamma(t_1),t_1) \leq (\gamma(t_2),t_2)\; \forall \:t_1 \leq t_2 \in I$. We now study the probability that the restriction of $T_n$ to this path is close to $e$. Let $u_{r,r'} := (\gamma(r), r; \gamma(r'), r')$ for $r< r'$, and consider the event 
\begin{equation} \label{heuristic}
	\cA := \bc{\abs{T_n(u_{r,r'}) -e(u_{r,r'})} \leq \delta}, 
\end{equation}
for a fixed $\delta >0$. By Property \ref{A2}, the probability of this event is approximately
\[\exp\br{-n J_m \br{\frac{e(u_{r,r'}) - \delta}{r' -r}}(r'-r)},\]
where $[x]_+ = \max (x,0),\, m = \frac{\gamma(r') -\gamma(r)}{r'-r}$. For $u=(x,s;y,t) \in \oR$, we define
\begin{align*}
	\Theta(e, u) := J_m \br{\frac{e(u)}{t -s}}(t-s), && \mathrm{for} \; m = \frac{y-x}{t-s}.
\end{align*}
By temporal independence of last passage times, the events in \eqref{heuristic} are independent across disjoint time intervals. Hence, given a collection $\{[s_i, t_i]\}_{i=1}^k$ of disjoint closed intervals in $I$, we have that the probability of event $\cA$ is  bounded above by
\[
	\exp\br{-n \sum_{i=1}^k \Theta(e - \delta, u_{s_{i}, t_i})}. 
\]  
The above equation holds for all countable collection of disjoint closed intervals of $I$; denote the set of all such collections by $\cC$.  Optimizing over all such collections and letting $\delta \to 0$ at an appropriate rate, the probability for the metrics to be close restricted to the path $\gamma$ can be upper bounded by
\begin{align} \label{eq: gdrf}
	\exp \br{-n \Theta(e, \gamma)}, && \mathrm{with}\; \Theta(e, \gamma) := \sup_{\cP \in \cC} \sum_{[s_{i}, t_i] \in \cP} \Theta(e, u_{s_{i}, t_i}). 
\end{align}

This comparison restricts us to a singular path, giving us a weak upper bound for the actual probability. Notice that if $u_i =(p_i;q_i) \in \oR, i =1,2$ are such that $\nexists\: r \in \bR^2$ such that $p_i \leq r \leq q_i$ for $i=1,2$ simultaneously, then the passage times $T_n(u_i)$ are independent. If a pair of points $u_i, i=1,2$ satisfy this condition then we shall call them disjoint, and we denote it by $u_1 \triangledown u_2$. We define the preliminary rate function $\Theta$ as taking the supremum over countable collection of such points in $\oR$,
\begin{align}
	\Theta(e) := \sup_{\substack{(u_i), s_i \neq t_i\\ u_i \triangledown  u_j, i \neq j}} \sum_{i=1}^k \Theta(e, u_i).
\end{align}
The next few sections are devoted to understanding how this functions assigns values to metrics and to developing the true rate function for our LDP.

\section{Finite Rate Metrics} \label{sec: frm}
This section is devoted to establishing the class of functions to which we shall assign finite rate. These functions have properties closely resembling those of directed metrics (see \cite{DV21} for definition); hence we refer to them as finite rate metrics. Note that strictly speaking, these functions are not directed metrics as the distance of a point to itself can be non-zero. 
 
\begin{definition} \label{def: D}
	For $m>0$, let $\cD_m \subseteq \cE$ be the set of functions $e$ satisfying the following conditions:
	\begin{enumerate}[(i)]
		\item $e$ is $d-$dominant: $e\geq d$.
		\item $e$ is $d-$close: $\Theta(e) \leq m$.
		\item $e$ satisfies a generalized metric composition law: for any points $(x,s) < (y,t)$ and $s \leq r < t$, 
		\begin{equation} \label{eq: metcomp}
			e(x,s;y,t) = \sup_{(x,s) \leq (z,r) < (y,t)} e(x,s; z,r) + e((z,r)^+;y,t),
		\end{equation}
		where $e(p^+,q)$ is as defined in \eqref{eq: triangle}.
	\end{enumerate}
\end{definition}
We set $\cD := \cup_{m \geq 0} \cD_m$, which shall contain all our finite rate metrics. Metrics in $\cE$ may allow exceptional points $p \in \bR$ such that $e(p,p) > 0$. We shall call such a point an atom. These exceptional points will be a deterrent in many of our arguments, however we have some workarounds for them. The following simple lemma will be particularly useful in this regard.
\begin{lemma} \label{lemma: ctble}
	Let $e \in \cE$ be such that $\Theta(e) < \infty$, then the set of atoms in $e$ is countable.
\end{lemma}
\begin{proof}
	Let $p$ be an atom, $u =(p,p)$. Since $e \geq 0$, by the triangle inequality \eqref{eq: triangle},
	\[e(p, p +(0,\epsilon)) \geq e(u),\]
	for all $\epsilon >0$. Since $d(p,p) =0$, we can choose $\epsilon$ small enough such that
	\[\Theta(e, (p, p +(0,\epsilon))) > 0.\]
	Thus, each atom can make a positive contribution to $\Theta(e)$. Assume that there are uncountable many atoms, one can now easily show that $\Theta(e)$ cannot be bounded by using elementary arguments.
\end{proof}
\begin{remark}
	We could have excluded metrics with have atoms from our discussion. However, elementary calculations show that such metrics will have finite rate when the point to point rate functions exhibit linear growth. This happens when the Cramer's rate function corresponding to the random weights is bounded by a linear function. Thus, these ill-behaved metrics get assigned finite rates even in the popularly studied exponential LPP and excluding them would unnecessarily limit the scope of our discussion. 
\end{remark} 

We can interpret $e(p^+,q)$ as the supremum of $\limsup_n e(p_n, q)$ over sequences $(p_n) \subset \vee_p \setminus \bc{p}$, converging to $p$. In fact, it suffices to look at the limit along sequences along the boundary of $\vee_p$. We formalize our assertion in the following lemma and leave the proof to the reader as an exercise.
\begin{lemma} \label{lemma: atoms}
	Let $e \in \cE$, $ u = (p;q) \in \oR$, and let 
	\[
		p_n^1 := p + (1,1)n^{-1} \;\;\text{and}\;\; p_n^2 := p+(-1,1)n^{-1}.
	\]
	Then for $p \neq q$, 
	 \[e(p^+; q) = \max \bc{\sup_n e(p_n^*; q) \;\vert\; * \in \bc{1, 2}}\] 
	 and it is zero otherwise.
\end{lemma}

We now collect some structural properties of $\cD$ and examine how hypo-convergence interacts with the structure of our finite rate metrics. Most importantly, we shall show that the sets $\cD_m$ are compact, and the finite rate metrics define geodesic spaces. We begin by showing that for $e \in\cD$, the supremum in \eqref{eq: metcomp} is always achieved. The arguments we present will be useful later, hence we make the statement in a greater generality.

\begin{lemma} \label{lemma: max}
	Let $e \in\cE$, $R_e \subseteq \bR$ be the maximal set such that there are no atoms in $\bR \times R_e$. Let $e$ satisfy the following properties
	\begin{enumerate} [(i)]
		\item $\Theta(e)$ is finite.
		\item For all $(p; q) = (x,s;y,t) \in \oR$ such that $p$ is not an atom, we have that 
		\begin{equation} \label{eq: max}
			e(x,s;y,t) = \sup_{(x,s) \leq (z,r) < (y,t)} e(x,s; z,r) + e((z,r)^+;y,t),
		\end{equation}	
		for $r \in [s, t) \cap R_e$.
	\end{enumerate} 
	Then $e$ satisfies Definition \ref{def: D}(iii) and the supremum is always achieved.
\end{lemma}
\begin{proof}
	Since $\Theta(e) \leq \infty$, by Lemma \ref{lemma: ctble}, $e$ can only have countably many atoms. Thus, the set $R_e$ must be dense in $\bR$. Let $u =(p;q) =(x,s;y,t)$, we claim that $e(p^+, q) = e(p,q)$ if $p$ is not an atom. This property holds trivially for metrics in $\cD$ by applying the metric composition law at $r =s$. The denseness of $R_e$ along with our second assumption shall allow us to prove it for our metric $e$ as well. Let $\epsilon>0$ and $(r_n) \subset R_e$, such that $r_n \downarrow s$. By our second assumption, there exists a sequence $z_n$ such that
	\[
		e(p,q) - \epsilon \leq e(p, (z_n, r_n)) + e((z_n, r_n)^+, q).
	\]
	Moreover, we have that $z_n$ must converge to $x$ as $r_n \downarrow s$. Thus by the upper semi-continuity of $e$, we have that
	\begin{align*}
		e(p,q) - \epsilon &\leq \limsup_n e(p, (z_n, r_n)) + e((z_n, r_n)^+, q) \\
		&\leq e(p, p)  + e(p^+, q).
	\end{align*}
	Since the above inequality holds for all $\epsilon >0$ and $p$ is not an atom, we have that
	\[e(p,q) \leq e(p^+, q).\]
	The reverse inequality follows from the generalized triangle inequality, proving our claim.

	Fix some endpoints $(p,q) = (x,s;y,t) \in \oR$ satisfying the assumptions in (ii). Let $r \in [s,t) \cap R_e$ and $I_r := \bc{z: p < (z,r) < q}$, be the set of values of $z$ which appear in the supremum of \eqref{eq: max} for a fixed $r$. As $(z,r)$ is not an atom, \eqref{eq: max} can be rewritten as
	\begin{equation} \label{eq: max1}
		e(p,q) = \sup_{z \in I_r} e(p; (z,r)) + e((z,r);q).
	\end{equation} 
	The maps $z \to e((z,r);q), \: z \to e(p; (z,r))$ defined on $I_r$ are both upper semi-continuous. So there sum will be as well and the supremum in the R.H.S. must be achieved at some point.
	
	We now pick arbitrary $(p,q) \in \oR$ and $r \in [s,t) \cap R_e$. Let $p_n = (x, s_n)$ and $q_n = (y, t_n)$ where $(s_n), (t_n) \subset R_e$, $s_n \uparrow s$ and $t_n \downarrow t$. By monotonicity of $e$, we have that $e(p_n, q_n) \geq e(p,q)$. Let $z_n$ denote the maximizer of \eqref{eq: max1} for the pair $(p_n , q_n)$ at $r$. If needed, we can pass through a subsequence so that $z_n$ is convergent and let $z_0$ be the limit. Note that $z_0$ must belong to the interval $I_r$. Hence,
	\begin{align*}
		e(p, q) &\leq \limsup_n e(p_n, q_n) \\
		&= \limsup_n e(p_n, (z_n,r)) + e((z_n,r), q_n) \\
		&\leq e(p, (z_0,r)) + e((z_0,r), q). 
	\end{align*}
	Since $e \in \cE$, we also have the reversed inequality.

	Now let $r \in (s,t) \setminus R_e$. By denseness of $R_e$, we can pick a sequence $r_n \downarrow r$ such that $r_n \in R_e$. Let $z_{r_1}$ be the rightmost maximizer for \eqref{eq: max1} at $r = r_1$. We define $z_{r_n}$ for $n \geq 2$, as the rightmost maximizer of 
	\[e(p; (z,r_n)) + e((z,r_n);(z_{r_{n-1}}, r_{n-1})).\]
	Now, by the second assumption of the lemma and the triangle inequality
	\begin{align*}
	e(p;q) &= e(p; (z_{r_1},r_1)) + e((z_{r_1},r_1);q) = e(p; (z_{r_2},r_2)) + e((z_{r_2},r_2);(z_{r_{1}}, r_{1}))+ e((z_{r_1},r_1);q) \\
	&\leq e(p; (z_2,r_2)) + e((z_2,r_2);q) \leq e(p;q).
	\end{align*}
	Hence, all the above inequalities must be equalities and repeating this argument shows that $z_{r_n}$ is a maximizer of \eqref{eq: max1} at $r_n$. If needed, we can pass through a subsequence of $r_n$ such that $z_{r_n}$ is convergent, call the limit $z_0$. The monotonicity of $e$ tells us that
	\begin{align*}
		e(p;q) &=  \limsup_n e(p; (z_{r_n},r_n)) + e((z_{r_n},r_n); q) \\ 
		&\leq e(p; (z_0,r)) + e\br{(z_0,r)^+ ; q},
	\end{align*}
	following the upper semi-continuity of $e$ and \ref{def: plus}.
\end{proof}

We are now ready to unravel the topological properties of our function space.
\begin{lemma} \label{Arz}
	The topological set $(\cE, \tau_h)$ is compact. 
\end{lemma}
\begin{proof}
	By Theorem \ref{thm: AA}, it suffices to show that $\cE$ is closed. Let $\br{e_n} \in \cE$ such that $e_n \to e$. The limit $e$ is non-negative since $e_n$ are non-negative, thus we only need to prove that $e$ satisfies the triangle inequality \eqref{eq: triangle}. Fix a point $u = (p,q) \in \oR$, and $r \in (s,t)$ with $z$ such that $p < (z,r) < q$. By \ref{hypo2}, there exists a sequence $(p_n, q_n) \to (p,q)$ such that 
	\[\lim e_n(p_n, q_n) = e(p,q).\]
	We can pick a sequence $\bc{t_n} \in \bR^{>0}, t_n \downarrow 0,$ such that $p'_n := p -(0,t_n) \leq p_n$ and $q'_n := q+(0,t_n) \geq q_n$ for all $n$. By monotonicity of $e_n$ and \ref{hypo1}, we have that 
	\begin{equation} \label{eq: cpt1}
		\lim e_n(p'_n, q'_n) = e(p,q).
	\end{equation} 

	We define $(z,r)_K = (z,r) + (1,1) K^{-1}$ for $K \in \bN$. By the same arguments as above, we can pick a sequence $\bc{t_{k}}$ such that \eqref{eq: cpt1} holds for the point pairs $(p,q), (p, (z,r))$ and $((z,r)_K, q)$. Thus, $t_k >0$, $t_k \downarrow 0$ and
	\begin{align*}
		&\lim e_n(p -(0,t_n), q +(0,t_n)) = e(p,q),\\
		&\lim e_n(p -(0,t_n), (z,r) +(0,t_n)) = e(p,(z,r)),\\
		\lim e_n&((z,r)_K -(0,t_n), q +(0,t_n)) = e((z,r)_K, q).
	\end{align*}
	Hence, we have that
	\begin{align*}
		e(p,q) &= \lim_n e_n(p -(0,t_n), q +(0,t_n)) \\
		&\geq \lim_n e_n(p -(0,t_n), (z,r) +(0,t_n)) + e_n(((z,r) -(0,t_n))^+, q +(0,t_n)) \\
		&\geq \lim_n e_n(p -(0,t_n), (z,r) +(0,t_n)) + e_n((z,r)_K -(0,t_n), q +(0,t_n)) \\
		&= e(p,(z,r)) + e((z, r)_K, q). 
	\end{align*}
	The above arguments also hold for the sequence $(z,r) + (-1,1) K^{-1}$. Letting $K \to \infty$ and applying Lemma \ref{lemma: atoms} gives us that
	\[
		e(p;q) \geq e(p;(z,r)) +e((z,r)^+, q).
	\]
\end{proof}

\begin{remark} \label{rmk: hypo-bound}
	In the proof of the above theorem, we have shown that if $e_n \to e$ in the $\tau_h$ topology, then for any pair of endpoints $u \in \oR$, $\epsilon > 0$, we have that for $n$ large enough 
	\[
		e(u) \leq e_n(u') + \epsilon,
	\]
	where $u' := u +(0, -\epsilon, 0, \epsilon)$. 
\end{remark}

\begin{lemma} \label{lemma: lsc}
	The function $e \longmapsto  \Theta(e)$ is lower semi-continuous for $e \in \cE$.
\end{lemma}
\begin{proof}
	Let $\bc{e_n}$ be a sequence converging to $e$. Let $u_i \in \oR,$ for $ i=1, \cdots, k$ be a set of disjoint points. By Remark \ref{rmk: hypo-bound}, $\bd(e_n,e) \to 0$ implies that for $\forall \epsilon > 0$ we have that  
	\begin{align*}
		e(u_i) \leq e_n(u'_i) + \epsilon,
	\end{align*}
	for $n$ large enough. We now choose $\epsilon$ small enough so that $u'_i$ are still disjoint and are in $\oR$. For a fixed $i$, we have that
	\begin{align*}
		\Theta(e, u_i) = J\br{\frac{y_i-x_i}{t_i -s_i}, \frac{e(u_i)}{t_i-s_i}}(t_i-s_i)&& \mathrm{and}&& \Theta(e_n, u'_i) = J\br{\frac{y_i-x_i}{t_i' -s_i' }, \frac{e_n(u'_i)}{t_i'-s_i'}}(t_i'-s_i').
	\end{align*}
	
	Property \ref{A2} states that the function $J$ is jointly continuous. Thus, for all $\delta>0$ there exists $\epsilon >0$ such that for all $i = 1, \cdots, k$ and $n$ large enough
	\begin{align*}
		\Theta(e&,u_i) \leq \Theta(e_n,u'_i) + \delta\\
		\Rightarrow \sum_{i=1}^k \Theta(e,u_i) \leq &\sum_{i=1}^k \Theta(e_n,u'_i) + \delta k \leq \Theta(e_n) + \delta k.
	\end{align*}
	Now, we make $\delta \to 0$ by letting $n \to \infty$. Thus, we have that
	\begin{align*}
	\sum_{i=1}^k \Theta(e,u_i) &\leq \liminf_n \Theta(e_n) \\
	\Rightarrow \Theta(e) &\leq \liminf_n \Theta(e_n).
	\end{align*}
\end{proof}

\begin{lemma} \label{lemma: Dm cpt}
	The sublevel sets $\cD_m$ are compact under the topology $\tau_h$. 
\end{lemma}
\begin{proof}
	By Lemma \ref{Arz}, it suffices to show that $\cD_m$ is closed. Thus, we need to show that the limit $e$ of any convergent sequence $\br{e_n} \in \cD_m$ must satisfy the properties stated in Definition \ref{def: D}. For any point $u \in \oR$, by Theorem \ref{hypo1}, $e(u) \geq \limsup_n e_n(u) \geq d(u)$. Hence, $e$ is $d-$dominant. By lower semi-continuity of $\Theta$, we have that $\Theta(e) \leq m$. Thus, all we need to show is that the metric composition law is closed under hypo-convergence. By Lemma \ref{Arz}, we have that $\cE$ is closed. Thus, $e \in \cE$, which implies that
	\begin{equation} \label{eq: mcl}
		e(p,q) \geq \sup_{\bc{z: p < (z,r) < q}} e(p, (z,r)) + e((z,r)^+, q),
	\end{equation}
	for $ (p,q) \in \oR$, and $r \in (s,t)$. 
	
	Assume that the equality in \eqref{eq: mcl} is not achieved for some choice of $(p,q) \in \oR$ and $r \in R_e$ which satisfy the second assumption of Lemma \ref{lemma: max}. Let $I_r := \bc{z: p < (z,r) < q}$ as in the proof of Lemma \ref{lemma: max}. Let $p_n = (x, s_n)$ and $q_n = (y, t_n)$ where $(s_n), (t_n) \subset R_e$, $s_n \uparrow s$ and $t_n \downarrow t$. As in the proof of Lemma \ref{Arz}, can choose the sequence in such a way so that
	\[
		e(p,q) = \lim_n e_n(p_n, q_n).
	\]
	By Lemma \ref{lemma: max}, we can pick $z_n$ such that
	\[
		e_n(p_n, q_n) = e_n(p_n; (z_n, r)) + e_n((z_n,r)^+; q_n).
	\]
	If needed, we can pass through a subsequence so that $z_n$ is convergent and let $z_0$ be the limit. Note that $z_0$ must belong to the interval $I_r$. Hence,
	\begin{align*}
		e(p, q) &= \limsup_n e_n(p_n, q_n) \\
		&= \limsup_n e_n(p_n, (z_n,r)) + e_n((z_n,r)^+, q_n) \\
		&\leq \limsup_n e_n(p_n, (z_n,r)) + e_n((z_n,r), q_n) \\
		&\leq e(p, (z_0,r)) + e((z_0,r), q). 
	\end{align*}
	Thus, $e$ must satisfy the assumptions of Lemma \ref{lemma: max}, and we have that $e \in \cD_m$. 
\end{proof}
\begin{corollary} \label{corollary: mcl}
	Let $e \in \cE$, such that $\Theta(e) < \infty$ and $e$ does not satisfy Definition \ref{def: D}(iii), then there exists $\delta_e >0$ such that $\bd(f,e) \leq \delta_e$ implies that $f$ also does not satisfy \ref{def: D}(iii). 
\end{corollary}
\begin{proof}
	Assume otherwise, then we have a sequence $f_n \to e$ such that all $f_n$ satisfy \ref{def: D}(iii). By Lemma \ref{lemma: Dm cpt}, we have that $e$ must satisfy \ref{def: D}(iii) as well.
\end{proof}
Now, let $e \in\cE$, recall that a continuous curve $\gamma :[s,t] \to \bR$ is called a curve in $\oR$ if $\bar{\gamma}(t_1) := (\gamma(t_1),t_1) \leq (\gamma(t_2),t_2)\; \forall \:t_1 \leq t_2 \in [s,t]$. We shall now define the $e-$length of a path. For a partition $\cP = \bc{s = r_0 < \cdots < r_k =t}$ of $[s,t]$, let 
\begin{align*}
	\abs{\gamma}_{e, \cP} = e(\bg(s), \bg(r_1)) + \sum_{i=2}^k e(\bg(r_{i-1})^+, \bg(r_i)), && \mathrm{and}&& \abs{\gamma}_e = \inf_{\cP \text{ a partition of } [a,b]} \abs{\gamma}_{e, \cP}.
\end{align*}
A path from $p$ to $q$ is a geodesic if $\abs{\gamma}_e = e(p, q)$. In general the length of a path will be less than the metric distance between the endpoints by the reversed triangle inequality. If for every $(p,q) \in \oR$, there is a geodesic from $p$ to $q$, then the metric $e$ is said to define a geodesic space. A geodesic $\gamma$ from $p$ to $q$ is called a rightmost geodesic if for every geodesic $\gamma'$ from $p$ to $q$, $\gamma \leq \gamma$ when viewed as functions of the time coordinate. We shall now establish that every element $e \in \cD$ defines a geodesic space. The argument is very similar to the one presented in \cite[Lemma 4.5]{DDV24}, with minor tweaks to fit our setting better.

\begin{lemma} \label{lemma: gd1}
	For every $e \in \cD,$ $u=(p,q) = (x,s;y,t) \in \oR$, and $r \in [s,t]$, the function
	\begin{equation} \label{eq: gd1}
		z \longmapsto e(p; z,r) + e((z,r)^+;q),
	\end{equation}
	has a rightmost maximizer $z_u(r)$, which is a continuous function of $r$.
\end{lemma}
\begin{proof}
	We have the existence of maximizers from Lemma \ref{lemma: max} and since $z_u(r)$ can only vary over a compact set, we also have the existence of rightmost maximizers.

	We shall now prove that $z_u$ is continuous. Since $z_u(r) \in I_r$, we have that $z_u$ is continuous at $s,t$. The upper semi-continuity of $e$ implies that $z_u$ is upper semi-continuous everywhere. By the metric composition law, for any $r' < r \in (s,t)$, $z_{(\bar{z}_u(r');q)}(r)$ maximizes the function in \eqref{eq: gd1}, hence $z_{(\bar{z}_u(r');q)}(r) \leq z_u(r)$. Therefore,
	\[z_u(r') = \liminf_{r \downarrow r'} z_{(\bar{z}_u(r');q)}(r) \leq \liminf_{r \downarrow r'} z_u(r),\]
	where the equality follows from the continuity of $z_{(\bar{z}_u(r');q)}(r)$ at its endpoints. Combining this with the upper semi-continuity of $z_u$, establishes the right-continuity of $z_u$. A symmetric argument gives left-continuity.
\end{proof}
Notice that Lemma \ref{lemma: gd1} is identical to \cite[Lemma 4.5]{DDV24}. Because of this identical form, we can directly apply \cite[Lemma 4.6, 4.7]{DDV24} to conclude the following properties.

\begin{lemma}[{\cite[Lemma 4.6]{DDV24}}] \label{lemma: gd2}
	For $e \in \cD$, and points $(p,q) = (x,s;y,t)$, $(p',q') = (x',s';y',t')$ with $x < x'$ and $y < y'$, we have
	\[e(p,q') + e(p',q) \leq e(p,q) + e(p',q'),\]
	whenever all the quantities in the above inequality are well-defined, i.e. $p \leq q, p \leq q', p' \leq q$ and $p'\leq q'$.  
\end{lemma}

\begin{lemma}[{\cite[Lemma 4.7]{DDV24}}] \label{lemma: gdf}
	The function $z_u$ defines a rightmost geodesic from $p$ to $q$.
\end{lemma}

This establishes our finite rate metrics as geodesic spaces. We shall see in the next section that the structure of the geodesics completely determines the metric.

\section{Networks and Metrics} \label{sec: ntwrk}

In this section, we develop the true rate function $I$ and understand the structure of finite rate metrics in $\cD$. A countable collection of closed paths in $\oR$, $\Gamma$ is a \textbf{network} if all paths in it are disjoint, i.e. for all $\gamma, \gamma' \in \Gamma$ with domains $[a,b], [a', b']$ and $r \in [a,b] \cap [a', b']$, we have $\gamma(r) \neq \gamma'(r)$. Let $\cP$ be a countable collection of disjoint intervals in an interval $S$. Define, $\cC_{S}$ to be the collection of all possible $\cP$. We define the mesh size of $\cP$, $m(\cP)$ as the supremum of the length of the intervals in $\cP$. The rate of a path $\gamma: [s,t] \to \bR^2$ is defined as 
\begin{equation}
	I(e, \gamma) := \inf_{\epsilon >0} \sup_{\cP \in \cC_{[s,t]}, m(\cP) \leq \epsilon} \sum_{[s_{i}, t_i] \in \cP} \Theta(e, (\bg(s_i), \bg(t_i))).
\end{equation}
We note that the supremum is non-increasing as $\epsilon$ decreases. The rate function can be similarly defined for curves with open/\,half-open intervals as domain. Finally, for $e \in \cD$, we define the rate function $I$ as
\begin{align} \label{eq: gdrf2}
	I(e) := \sup_{\Gamma} I(e, \Gamma),&& \mathrm{where} && I(e, \Gamma) := \sum_{\gamma \in \Gamma} I(e, \gamma),
\end{align} 
and set $I(e) = \infty$ for $e \notin \cD$.

We note that the supremum in \eqref{eq: gdrf2} can be restricted to finite networks. We shall prove the equivalence between $\Theta$ and $I$ in the Lemma \ref{lemma: equiv}. Before doing so, we make the following important observation.

\begin{lemma} \label{lemma: ugd}
	Let $e \in \cD$, $u =(x,s;y,t) \in \oR$ and $\gamma$ be an $e-$geodesic with endpoints $(x,s)$ and $(y,t)$. Then
	\[
		\Theta(e,u) \leq I(e, \gamma).
	\] 
\end{lemma}
\begin{proof}
	Let $\epsilon > 0$. Since $\gamma$ is the geodesic corresponding to $u$, we can pick a finite collection $\cP \in \cC_{[s,t]}$ with $m(\cP) \leq \epsilon$, such that
	\begin{equation} \label{eq: ugd1}
		\sum_{[a_i,b_i] \in \cP} e((\bg(a_i), \bg(b_i))) \geq e(u) - \delta,
	\end{equation}	
	for arbitrarily small $\delta >0$. Let $\eta >0$, we shall show that for $\delta$ small enough
	\begin{equation} \label{eq: ugd2}
		\sum_i \Theta(e, (\bg(a_i), \bg(b_i))) \geq \Theta(e,u) - \eta.
	\end{equation}
	Let $u_i = (\bg(a_i), \bg(b_i))$ and define the events 
	\[
		A_i = \bc{T_n(u_i) \geq \abs{\gamma_{[a_i, b_i]}}}.
	\] 
	Following \eqref{eq: ugd1}, we have that
	\[
		\bigcap_i A_i \subset A:=\bc{T_n(u) \geq e(u) -\delta}.
	\] 
	Applying Property \ref{A2}, we have that
	\[
		\P{\bigcap_i A_i} \leq \P{A} \leq \exp \br{-n \Theta(e - \delta,u)} \leq \exp \br{-n (\Theta(e,u) - \eta)},
 	\]
	for $\delta$ small enough, as $J$ is continuous. At the same time, by the independence property of $T_n$ 
	\[
		\P{\bigcap_i A_i} = \prod_i \P{A_i} \geq \exp \br{-n\sum_i(\Theta(e,u_i) +o(1))},
	\] 
	where the inequality follows from Property \ref{A2} for $n$ large enough as $u_i$ are only finitely many. Taking the logarithm and dividing by $n$ in the above equations, we have proved that \eqref{eq: ugd2} holds. This shows that
	\[
		I(e,\gamma) \geq \Theta(e,u) - \eta, \; \forall \eta >0.
	\]
\end{proof} 

The arguments in the proof of Lemma \ref{lemma: ugd} are fairly general and lead to the following corollary.

\begin{corollary} \label{cor: path-approx}
	Let $e \in \cD$ and $u = (p;q) \in \oR$. Let $\epsilon >0$ and $\bc{u_i = (p_i;q_i) \in \oR}_{i=1}^n$ be a finite collection of points such that $p_1 \geq p, p_{i+1} \geq q_i$, $q_n \leq q$ and
	\[
		\sum_i e(u_i) \geq e(u) - \epsilon.
	\]
	Then,
	\[
		\liminf_{\epsilon \downarrow 0}\sum_i \Theta(e, u_i) \geq \Theta(e, u).
	\]
\end{corollary}

\begin{lemma} \label{lemma: equiv}
	For any $e \in \cD$, we have that $\Theta(e) = I(e)$.
\end{lemma}
\begin{proof}
	Assume that we have a finite disjoint network $\bc{\gamma_i}$. Since $\gamma_i$ is a disjoint collection of closed curve, there exists $\epsilon>0$ such that for $i \neq j$, 
	\begin{equation} \label{eq: equiv1}
		d^1(\gamma_i, \gamma_j) > \epsilon.
	\end{equation}
	Following the definition of $I$, the supremum in $I(e,\gamma_i)$ over countably many disjoint intervals, each with length less than $\epsilon$, is greater than or equal to $I(e, \gamma_i)$. Each interval corresponds to a point in $\oR$ and \eqref{eq: equiv1} ensures that the collection of all points from each $\gamma_i$ is disjoint. Hence, for any finite network $\Gamma$ we have that
	\[\Theta(e) \geq I(e, \Gamma) .\]
	Thus, $\Theta(e) \geq I(e)$.

	Given any countable collection of disjoint points $U = \bc{u_i = (p_i, q_i)}$, by Lemma \ref{lemma: gdf} there is a rightmost geodesic $\gamma_i$ with endpoints $p_i$ and $q_i$. Since $u_i$ are disjoint, so are their geodesics. Thus, $\Gamma = \bc{\gamma_i \: |\: u_i \in U}$ is a network and by Lemma \ref{lemma: ugd}
	\[I(e, \Gamma) \geq \sum_{i} \Theta(e,u_i).\] 
	Hence, $I(e) \geq \Theta(e)$. 
\end{proof}
The proof of Lemma \ref{lemma: equiv} shows that for $e \in \cD$, we can take the supremum in \eqref{eq: gdrf2} over collections of disjoint geodesics as well. We now show that $I$ is lower semi-continuous, hence a valid candidate for a rate function.

\begin{corollary} \label{cor: lsc}
	The function $e \longmapsto I(e)$ is lower semi-continuous for $e \in \cE$.
\end{corollary}
\begin{proof}
Let $(e_n) \subset \cE$, such that $e_n \to e$ for some $e \in \cD$. By our definition of $I$ and Lemma \ref{lemma: equiv}, we have that $I \geq \Theta$ on $\cE$. By Lemma \ref{lemma: lsc}, we have that $\Theta$ is lower semi-continuous. Thus, we have that $\liminf_n I(e_n) \geq \liminf_n \Theta(e_n) \geq \Theta(e) = I(e)$.

Now assume that the limit point $e \notin \cD$, then by Lemma \ref{lemma: Dm cpt}, we must have that $\Theta(e_n)$ is unbounded or $e_n$ is not in $\cD$ for $n$ large enough. In either case, we have that $\liminf_n I(e_n) = \infty$, proving our claim.
\end{proof}

Our next goal is to show that for $e \in \cD$, the supremum in \eqref{eq: gdrf2} is achieved by a network of paths. Such networks will enable many arguments that we will be requiring later, and we shall call these special networks \textbf{$e-$complete}. 

\begin{lemma} \label{lemma: breakgd}
	Let $e \in \cD$. Given an $e-$geodesic path $\gamma: [s,t] \to \bR^2$, for any $x \in [s,t]$, we have that
	\[I(e, \gamma) = I(e, \gamma|_{[s,x]}) + I(e, \gamma|_{(x,t]}).\]
\end{lemma}
\begin{proof}
	Let $\cP_1 \in \cC_{[s,x]}, \cP_2 \in \cC_{(x,t]}$, then their union $\cP^1 \cup \cP^2 \in \cC_{[s,t]}$. Hence, we have that
	\[I(e, \gamma) \geq I(e, \gamma|_{[s,x]}) + I(e, \gamma|_{(x,t]}).\]

	Thus, all we need to show is that the reverse inequality also holds. Let $\epsilon >0$ and $\cP \in \cC_{[s,t]}$. It suffices to show that there exists a collection $\cP'$, such that any $I \in \cP'$ is either a subset of $[s,x]$ or $(x,t]$ with the property that
	\[\sum_{[s_{i}, t_i] \in \cP'} \Theta(e, (\bg(s_i), \bg(t_i))) \geq \sum_{[s_{i}, t_i] \in \cP} \Theta(e, (\bg(s_i), \bg(t_i))) - \epsilon.\]
	Let $I =[a,b] \in \cP$ such that $x \in I$. If no such $I$ exists, then $\cP' =\cP$ is our desired collection, so we assume otherwise. Consider the sets, $I_1 = [a,x]$ and $I_2 = [x+\delta, b]$. We claim that the collection
	\[\cP' = \cP \cup \bc{I_1, I_2} \setminus \bc{I}\]
	satisfies the claim for $\delta$ small enough. The only inequality to check is that
	\begin{equation} \label{eq: breakgd}
		\Theta(e, (\bg(a), \bg(b))) - \epsilon \leq \Theta(e, (\bg(a), \bg(x))) + \Theta(e, (\bg(x+ \delta), \bg(b))).
	\end{equation}
	Let $\eta >0$. Since $\gamma$ is a geodesic, we can choose $\delta >0$ such that
	\[
		e((\bg(a), \bg(b))) -\eta \leq e((\bg(a), \bg(x))) + e((\bg(x+ \delta), \bg(b))).
	\]
	The claim now follows from Corollary \ref{cor: path-approx}.
\end{proof}

\begin{lemma} \label{lemma: ntwrk-upgrd}
	Let $e$ be such that $I(e) \leq \infty$ and $\Gamma$ be a finite network, $\Pi$ be a finite network of geodesics. For any $\epsilon>0$, there is a finite network $\Gamma' \supset \Gamma$ such that
	\[
		I(e, \Gamma') \geq I(e, \Pi) - \epsilon.
	\]
	Furthermore, if $\Gamma$ is a geodesic network, then so is $\Gamma'$.
\end{lemma}
\begin{proof}
	Let $n$ be the cardinality of $\Pi$. For $\gamma \in \Gamma$, $\pi \in \Pi$, let $K_{\gamma, \pi}$ be the set of times $t$ for which $\gamma(t) = \pi(t)$. The set $K_\pi = \bigcup_{\gamma \in \Gamma} K_{\gamma, \pi}$ is closed, so $[s_\pi, t_\pi] \setminus K_\pi$ is a countable union of relatively open intervals of $[s_\pi, t_\pi]$. We denote the collection of these open intervals by $O_\pi$ and choose a finite subset $O'_\pi$ of $O_\pi$ such that 
	\[
		\sum_{I \in O'_\pi}I(e, \pi|_{I}) \geq \sum_{I \in O_\pi}I(e, \pi|_{I}) - \epsilon/n
	\] 
	Let $m_\pi$ be the cardinality of $O'_\pi$ and $I \in O'_\pi$. By the definition of our rate function, for each interval $I$ we have a finite collection of disjoint intervals $C_{\pi,I}$ such that
	\[\sum_{[a, b] \in C_{\pi, I}} \Theta(e, (\bar \pi(a), \bar \pi(b))) \geq I(e, \pi|_I) - \epsilon/nm_\pi.\]
	By Lemma \ref{lemma: ugd}, we have that 
	\[\sum_{[a, b] \in C_{\pi, I}} I(e, \pi|_{[a,b]})\geq I(e, \pi|_I) - \epsilon/nm_\pi.\]
	Hence, we have a finite collection of closed intervals $C_\pi = \cup _{I \in O_\pi'} C_{\pi, I}$ such that each interval in $C_\pi$ is contained in an interval of $O_\pi$ and the collection of paths $D_\pi$, which we get on restricting $\pi$ to $C_\pi$ satisfies
	\[I(e, D_\pi) \geq \sum_{I \in O_\pi}I(e, \pi|_{I}) - 2\epsilon/n.\]
	Let $\Gamma' = \Gamma \cup \bigcup_{\pi \in \Pi} D_\pi.$ By construction, $\Gamma'$ is a finite network of geodesic paths. Furthermore,
	\begin{align*}
		I(e, \Gamma') &= I(e, \Gamma) + \sum_{\pi \in \Pi} I(e, D_\pi) \\
		&\geq \sum_{\pi \in \Pi} \sum_{I \in K_\pi} I(e, \pi|_I) + \sum_{\pi \in \Pi} \br{\sum_{I \in O_\pi}I(e, \pi|_{I})} - 2\epsilon/n \\
		&= I(e, \Pi) - 2\epsilon,
	\end{align*} 
	where the last inequality follows by repeatedly applying Lemma \ref{lemma: breakgd}.
\end{proof}

\begin{proposition} \label{prop: e-comp}
	Let $e \in \cD$, then there exists an $e-$complete network $\Gamma$. In fact, every finite network is contained in an $e-$complete network. 
\end{proposition}
\begin{proof}
Consider a sequence of finite geodesic networks $\Pi_n$ such that 
\[I(e, \Pi_n) \geq I(e) - 1/n.\]
Let $\Gamma_1 =\Pi_1$ and define $\Gamma_n$ inductively as the network given by Lemma \ref{lemma: ntwrk-upgrd} which contains $\Gamma_{n-1}$ and  
\[I(e, \Gamma_n) \geq I(e, \Pi_n) - 1/n.\]
The network $\Gamma = \bigcup_n \Gamma_n$ is our desired network. To show the second part of our claim, we simply set $\Gamma_1$ as the given finite network.
\end{proof}

As a consequence of the above proposition, we show that our rate function $I$ is strictly monotone.
\begin{corollary} \label{cor: monotone}
	Let $e, e' \in \cD$ such that $e' \leq e$ and $e' \neq e$. Then $I(e') < I(e)$.
\end{corollary}
\begin{proof}
	Since $e' \neq e$, there exists an $e-$geodesic $\gamma$ with endpoints $u$ such that $\abs{\gamma}_e > \abs{\gamma}_{e'}$. By Property \ref{A2}, the point to point rate function is strictly increasing and hence
	\[
		I(e', \gamma) < I(e, \gamma).
	\]
	Let $\Pi =\bc{\gamma}$. Using Proposition \ref{prop: e-comp}, we can extend $\Pi$ to an $e'-$complete network $\Gamma$. Hence,
	\[
		I(e) \geq I(e, \Gamma) > I(e', \Gamma) = I(e'). \qedhere
	\] 
\end{proof}
Proposition \ref{prop: e-comp} will also allow us to fully understand the structure of metrics $e \in \cD$. In fact, knowledge about an $e-$complete network is enough to understand the whole metric $e$. We formalize this idea in the following proposition.

\begin{proposition} \label{prop: e-structure}
	Assume that we have a finite rate metric $e$ with an $e-$complete network $\Gamma$. Since paths in $\Gamma$ are disjoint, given any point $(p;q) = (x,s;y,t)$, there is at most one path $\gamma \in \Gamma$ such that $\gamma(s) =x$  and $\gamma(t) =y$.
	
	Hence, we can unambiguously define $e^0(p;q) = \abs{\gamma_{[s,t]}}_e$ and $e^0(p^+;q) = \abs{\gamma_{(s,t]}}_e$ if $\gamma(s) =x$ and $\gamma(t) =y$ for some $\gamma \in \Gamma$, and $e^0(p;q) = e^0(p^+;q) = 0$ otherwise. Let
	\[
		e'(p;q) = \sup_\cP \sup_\pi \;(e^0 \vee d)(\bar \pi(t_0),\bar \pi(t_1)) + \sum_{i = 1}^{l-1} (e^0 \vee d)(\bar\pi(t_i)^+, \bar\pi(t_{i+1})), 
	\]
	where the outer supremum is over partitions $s=t_0 < \cdots < t_l =t$ of $[s,t]$ and the inner supremum is over paths $\pi$ with endpoints $p$ and $q$. We claim that $e'=e$. Furthermore, the inner supremum is non-decreasing as one makes the partition finer.
\end{proposition}

\begin{proof}
	The refinement claim can be proven by induction. As we add extra time points, the inequality to check is 
	\[\sup_x (e^0\vee d)(p; x,r) + (e^0\vee d)((x,r)^+;q) \geq (e^0\vee d)(p;q),\]
	for $(x,r) \in \vee_p$ such that $\vee_{(x,r)} \ni q$. When $d(p;q) \geq e^0(p;q)$, we can take $x$ to be on the line segment connecting $p$ and $q$ and $d$ in both terms on the left. Otherwise, let $\gamma \in \Gamma$ so that $p,q$ lie on the curve $\gamma$ and take $x= \gamma(r)$ and $e^0$ in both terms on the left.

	It follows from definition that, $d \leq e' \leq e$. This implies that $\Theta(e') < \infty$. Thus, $e' \in \cD$ as \eqref{eq: metcomp} follows from the definition of $e'$. Since $\Gamma$ is an $e-$complete network and $e',e$ agree on $\Gamma$, we have that
	\[I(e') \geq I(e', \Gamma) =I(e, \Gamma) = I(e).\]
	By Corollary \ref{cor: monotone}, $e'$ must be equal to $e$.
\end{proof}
Metrics which have $e-$complete networks will be called planted network metrics. The terminology is motivated by the fact that the structure of the metric is completely determined by the network. The following corollary is a direct consequence of Proposition \ref{prop: e-structure}.

\begin{corollary} \label{cor: metric-approx}
	Any $e \in \cD$ can be approximated by finite planted network metrics.
\end{corollary}
\begin{proof}
	Let $\Gamma$ be an $e-$complete network. For all $n \in\bN$, we can find a finite network $\Gamma_n \subset \Gamma$ such that $I(e, \Gamma_n) \geq I(e) - 1/n$ and $\Gamma_n \subset \Gamma_{n+1}$. By Proposition \ref{prop: e-structure}, using each network $\Gamma_n$ we can construct a finite rate metric $e_n$ such that
	\begin{align*}
		I(e_n)\geq I(e) - 1/n && \mathrm{and} &&	e_n \leq e_{n+1} \leq e.		
	\end{align*}
	Since $\cD$ is compact, we can find a subsequence $e_{n_k}$ converging to some $e' \in \cD$. By Theorem \ref{hypo1}, we have that $e'(u) \geq \limsup_k e_{n_k}(u)$ and $e' \leq e$. Since $e_{n_k}$ is a non-decreasing sequence, we have that $e' \geq e_{n}$ for all $n$. Hence, $I(e') \geq I(e)$. Combining this with Corollary \ref{cor: monotone}, we have that $e' = e$. 
\end{proof}

\section{Proof of Upper Bound} \label{sec: pf-ub}
In the next few lemmas, we gather the tools needed to prove the upper bound for the probability of the large deviation event. The strategy for proving the upper bound is identical to the one adopted in \cite[Section 7]{DDV24}; however the intermediate steps must be proved differently. 

Let $K_m := \bc{e \in \cE : \Theta(e) \leq m}$. Lower semi-continuity of $\Theta$ (Lemma \ref{lemma: lsc}) implies that $K_m$ is closed in $\cE$. Combined with Lemma \ref{Arz}, we have that $K_m$ is compact.

\begin{lemma} \label{lemma: kmub}
	The following bound holds for all $\delta > 0, m > 0$
	\begin{equation}
		\limsup_{n \to \infty}  \frac{1}{n} \log \P{\bd(K_m, T_n) \geq \delta} \leq -m. 
	\end{equation}
\end{lemma}
\begin{proof}
	Let $e \notin B_\delta(K_m)$, and fix an $\epsilon >0$.	By lower semi-continuity of $\Theta$, there exists an $\eta_e >0$ such that $$\inf_{\bd(f,e) < \eta_e} \Theta(f) \geq \Theta(e) - \epsilon/2.$$  
	The collection $\bc{B_{\eta_e}(e)}$ forms an open cover of the set $B_\delta(K_m)^c$. The set $B_\delta(K_m)^c$ is closed in $\cE$, hence compact. Thus, it has a finite subcover $\bc{B_{\eta_i}(e_i)}_{i \leq k}$. 

	On the event $\bc{\bd(K_m, T_n) \geq \delta}$, $T_n$ must lie in one of these open balls. Hence, there exists an index $i$ such that
	\[\bd(T_n, e_i) \leq \eta_i.\]
	This implies that $\Theta(T_n) \geq m - \epsilon/2$. Now, by definition of $\Theta$, $\exists u_1, \cdots, u_l$ such that $u_i \triangledown u_j,$ for $i \neq j$ and
	\[
		\sum_{i=1}^l \Theta(T_n, u_i) \geq m -\epsilon.
	\] 
	By the independence property of $T_n$ over disjoint points and Property \ref{A2},
	\[
		\P{\sum_{i=1}^l \Theta(T_n, u_i) \geq m} \leq \exp(-n(m -\epsilon)).
	\]
	Since, the bound holds for all $\epsilon >0$, we get our desired inequality.
\end{proof}

\begin{lemma} \label{lemma: Dm bnd}
	The following bound holds for all $\delta > 0, m > 0$
	\begin{equation}
		\limsup_n \frac{1}{n} \log \P{\bd(\cD_m, T_n) \geq \delta} \leq -m.
	\end{equation}
\end{lemma}
\begin{proof}
	Let $m, \delta >0$, and define
	\[C_{m, \delta} = K_m \setminus B_\delta(D_m).\]
	Hence, for every $e \in C_{m, \delta}$, at least one of the following properties must hold
	\begin{enumerate}
		\item There exists $u \in \oR$ such that $e(u) < d(u)$.
		\item There exists $(p,q) = (x,s;y,t) \in \oR$ and $r\in(s,t)$ such that 
		\[e(p,q) > \sup_{p < (z,r) < q} e(p, (z,r)) + e((z,r)^+, q).\]  
	\end{enumerate}
	We shall show that in both cases, $\exists \eta_e >0$ such that
	\begin{equation} \label{eq: Dm-ub}
		\P{d(T_n, e) < \eta_e} \leq (1+o(1))\exp \br{-nm}.
	\end{equation}
	Let $e$ satisfy (1), then $\exists \epsilon >0$ and $u =(x,s;y,t) \in \oR$ such that $e(u) \leq d(u) - 3\epsilon$.	Since $e$ is upper-semicontinuous and $d$ is continuous, we can choose $u$ such that $\abs{x-y} \neq t-s$. Again by upper semi-continuity of $e$, there exists an open ball $B \ni u$ such that $\sup \bc{e(w) : w \in B} < d(u) - 2 \epsilon$. Using Remark \ref{rmk: hypo-bound}, we can take $\eta$ small enough such that $\bd(T_n, e) \leq \eta$ implies that $T_n(u) \leq d(u) -\epsilon$. Now, by Property \ref{A3}, we have that,
	\[\P{d(T_n, e) \leq \eta} \leq \exp(-2nm),\]
	for $n$ large enough. If $e$ satisfies the second property, then Corollary \ref{corollary: mcl} implies that for $\bd(T_n, e) \leq \eta_e$, $T_n$ must also not satisfy the metric composition law. However, this event is not possible owing to the definition of $T_n$.
	
	Now, consider the open cover $\bc{B_{\eta_e}(e) : e \in C_{m,\delta}}$. By compactness of $C_{m, \delta}$, we have a finite subcover of open balls $\bc{B_{\eta_i}(e_i)}_{i=1}^k$. If $T_n \notin B_\delta(\cD_m)$, then either $T_n \notin B_\delta(K_m)$ or $T_n \in B_{\eta_i}(e_i)$ for some $i=1, \cdots, k$. The required bound for the probability now follows from Lemma \ref{lemma: kmub} and \eqref{eq: Dm-ub}.
\end{proof}

\begin{lemma} \label{lemma: smallballs}
	For any $e \in \cD$, we have
	\[
		\lim_{\delta \downarrow 0} \limsup_{n \to \infty} \frac{1}{n} \log \P{T_n \in B_\delta(e)} \leq - I(e).
	\]
\end{lemma}
\begin{proof}
	Fix an $\epsilon >0$. By lower semi-continuity of $\Theta$, we can choose $\delta$ small enough such that
	\[\Theta(e') \geq \Theta(e) -\epsilon/2,\]
	for all $e' \in B_\delta(e)$. Let $T_n \in B_\delta(e)$, we can now choose finitely many $u_i \in \oR, u_i \triangledown u_j$, for $i \neq j$ such that 
	\[\sum_i \Theta(T_n, u_i) \geq \Theta(e) -\epsilon.\]
	By the independence property of $T_n$, the probability of this event is upper bounded by 
	\[\exp(-n(\Theta(e) -\epsilon)).\]
	Since, the argument works for all $\epsilon >0$, we have that
	\[
		\lim_{\delta \downarrow 0} \limsup_{n \to \infty} \frac{1}{n} \log \P{T_n \in B_\delta(e)} \leq - \Theta(e).
	\]
	We now apply Lemma \ref{lemma: equiv} to get the desired inequality.
\end{proof}

The proof of the upper bound now follows immediately from these lemmas. 
\subsection*{Proof of Upper Bound}
	Fix any closed set $C \subseteq \cE$ and let $\delta, m >0$. By Lemma \ref{lemma: smallballs}, for every $e \in C \cap \cD_m$ we can find an open ball $B_e$ centered at $e$ such that
	\begin{equation} \label{eq: ubp}
		\limsup_{n \to \infty} \frac{1}{n} \log \P{T_n \in B_e} \leq - I(e) +\delta.
	\end{equation}
	By Lemma \ref{lemma: Dm cpt}, we have that $C \cap \cD_m$ is compact, so it is covered by a finite collection $\cO$ of such balls. Let $r$ be the distance between the compact set $C \cap \cD_m$ and the closed set $\br{\bigcup \cO}^c$. Thus,
	\[
		C \subset \br{\bigcup \cO} \cup \bc{e \in \cE : \bd(e, \cD_m) > r/2}.
	\] 
	Now by \eqref{eq: ubp} and Lemma \ref{lemma: Dm bnd}, we have that
	\[
		\limsup_{n \to \infty} \frac{1}{n} \log \P{T_n \in C} \leq -\min \br{\inf_{e \in C} I(e) -\delta, m}.
	\]
	Taking $m \to \infty$ and $\delta \to 0$ completes the proof.

\section{Proof of Lower Bound} \label{sec: pf-lb}
The next item on our agenda will be to prove the corresponding lower bound. Towards this purpose, we shall first prove the bound for certain subsets of $\cE$. Given a metric $e$, we define the cone with apex $e$ as
\[\ce := \bc{e' \in \cE : e' \geq e}.\]
\begin{theorem} \label{thm: cone}
	Let $e$ be a finite planted network metric, i.e. it has a finite complete network $\Gamma$. Then, for $\delta >0$
	\[\liminf_{n \to \infty} \frac{1}{n} \log \P{T_n \in B_\delta(\ce)} \geq - I(e).\]
\end{theorem}
\begin{proof}
	Assume that $I(e) = m < \infty$, otherwise there is nothing to prove. The idea of the proof is to reduce the problem to having control over finitely many values of $T_n$. We can write
	\begin{align*}
		\ce = \bigcap_{r \in (0, \infty), u \in \oR} G(u, e-r), && G(u,e) := \bc{e' \in \cE : e'(u) \geq e(u)}.
	\end{align*}
	The sets $G(u,e)$ are closed under the hypograph topology. Thus,
	\begin{align*}
		\bar{B}_{\delta}(K_{2m}) \cap B_{\delta}(\ce)^c \subset \ce^c = \bigcup_{r \in (0, \infty), u \in \oR} G(u, e-r)^c,
	\end{align*}
	is an open cover, where $K_{2m}$ is as defined in Section \ref{sec: pf-ub} and $\bar{B}_{\delta}(K_{2m})$ represents the closed ball $\bc{e \in \cE : d(e, K_{2m}) \leq \delta}$. Since $\cE$ is compact and the L.H.S.\, is a closed set, we can find a finite subcover $\bc{G(u, e-r_u)^c: u \in Q}$. Let $r = (\delta \wedge \min_{u \in Q} r_u)$, then $\bc{G(u, e-r)^c: u \in Q}$ is also a subcover, giving us 

	\[B_{\delta}(\ce) \supset \bar{B}_{\delta}(K_{2m}) \cap \bigcap_{u \in Q} G(u, e- r).\]  

	Now, by Lemma \ref{lemma: kmub} we have that for $n$ large enough
	\[
		\P{T_n \notin \bar{B}_{\delta}(K_{2m})} \leq \exp \br{-(2m-o(1))n}.
	\]
	Hence, it suffices to show that
	\begin{equation}\label{eq: conemain}
	\P{T_n \in \bigcap_{u \in Q} G(u, e- r)} \geq \exp \br{-n(m + o(1))}.
	\end{equation}
	We now fix a finite $e-$complete network $\Gamma$. Since $\Gamma$ is a finite disjoint collection of closed curves, there exists $\tau>0$ such that for any two distinct curves $\gamma, \gamma' \in \Gamma$, we have that 
	\begin{equation} \label{eq: cones1} 
		d^1(\gamma', \gamma) > \tau.
	\end{equation}
	Let $S = \bigcup_{(x,s;y,t) \in Q} \bc{s,t},$ the set containing all time coordinates in $Q$. By Proposition \ref{prop: e-structure}, we can find a partition $\cP = \bc{t_0 < t_1 <\cdots < t_k}$ containing $S$ so that the following holds. First, $t_{i+1} - t_i \leq \tau$ for all $i$. Second, for all $u =(x, t_i; y, t_j) \in Q$ there exists points $z_k \in \bR^2$ with time components $t_k$ such that $(x, t_i) = z_i, z_{i+1}, \cdots, z_j = (y, t_j)$ and
	\begin{align*}
		(e^0 \vee d)(z_i, z_{i+1}) + \sum_{k = i+1}^j (e^0 \vee d)(z_k^+, z_{k+1}) \geq e(u) - r/2.
	\end{align*}
	Here the function $e^0$ is defined as in Proposition \ref{prop: e-structure}. Using Lemma \ref{lemma: atoms}, we can pick $z_k' \in \vee_{z_k} \setminus \bc{z_k}$ so that
	\begin{align*}
		\sum_{v \in V_u} (e^0 \vee d)(v) \geq e(u) - r &&\mathrm{where}&& V_u =\bc{(z_k', z_{k+1}) \:|\: i \leq k < j}.
	\end{align*}
	Since $Q$ is a finite set, we can impose that $z_k'$ have time component $t_k'$ which is independent of $u$. Let $V = \cup_{u \in Q} V_u $ and apply the reversed triangle inequality for $T_n$ to get that
	\[
		\P{T_n \in \bigcap_{u \in Q} G(u, e- r)} \geq \P{\bigcap_{v \in V}\bc{T_n(v) \geq (e^0 \vee d)(v)}}.
	\]
	The FKG inequality tells us that,
	\begin{align*}
		\P{\bigcap_{v \in V}\bc{T_n(v) \geq (e^0 \vee d)(v)}}&\geq \prod_{v \in V} \P{T_n(v) \geq (e^0 \vee d)(v)} \\
		&= \prod_{\substack{v \in V,\\ e^0(v) \geq d(v)}} \P{T_n(v) \geq e^0(v)} \prod_{\substack{v \in V,\\ e^0(v) < d(v)}} \P{T_n(v) \geq d(v)}.
	\end{align*}
	By Property \ref{A2}, the second term will make a contribution of order $\exp (-n o(1))$. Thus, we only need to control the first term to complete our claim. Notice that the first term is the same as taking the product over $V' = \bc{(x,s;y,t) \in V : \exists\: \gamma \in \Gamma \;\mathrm{such}\; \mathrm{that}\; \gamma(s) = x \;\mathrm{and}\; \gamma(t) = y}$. By applying Property \ref{A2} and using the definition of $\Theta$, we get that
	\[
		\prod_{v \in V, e^0(v) \geq d(v)} \P{T_n(v) \geq e^0(v)}\geq \exp \br{-n \sum_{v \in V'} (\Theta(e, v) + o(1))}. 
	\]	
	Now for all $(x,s;y,t) \in V'$, we have that $t-s \leq \tau$. Equation \eqref{eq: cones1} thus forces $V'$ to be a collection of disjoint points. Thus, by the equivalence of $\Theta$ and $I$ (see Lemma \ref{lemma: equiv}), we have that
	\[\P{\bigcap_{v \in V}\bc{T_n(v) \geq (e^0 \vee d)(v)}} \geq \exp \br{-n(m + o(1))}.\]
\end{proof}

\subsection*{Proof of Lower Bound}
Let $O$ be open in $\cE$, if $O$ does not contain any finite rate metric then there is nothing to prove. Otherwise, let $e_0 \in O$ such that $I(e_0) < \infty$. By Corollary \ref{cor: metric-approx}, we can find a finite planted network metric $e$ such that $e \in O$ and $I(e)$ is arbitrarily close to $I(e_0)$. Thus, it suffices to show the following bound for any finite planted network metric $e \in O$,
\[
	\liminf_n \frac{1}{n} \log \P{T_n \in O} \geq -I(e).
\]
We bound $\P{T_n \in O}$ below by 
\[
	\P{T_n \in O \cap B_\delta(\ce)} = \P{T_n \in B_\delta(\ce)} - \P{T_n \in B_\delta(\ce) \cap O^c}.
\] 
In view of Theorem \ref{thm: cone}, it suffices to show that for some $\delta >0$,
\[
	\limsup_n \frac{1}{n} \log \P{T_n \in B_\delta(\ce) \cap O^c} \leq -I(e). 	
\]
The closure of $B_\delta(\ce) \cap O^c$ is contained in $B_{2\delta}(\ce) \cap O^c$. By the large deviation upper bound, we just need to check that
\[
	s = \sup _{\delta >0} \inf_{e' \in B_{2\delta}(\ce) \cap O^c} I(e') > I(e).
\]
If $s = \infty$, there is nothing to prove. Otherwise, we can find $e_n \in B_{1/n}(\ce)\cap O^c$ with $I(e_n) \to s$. The sublevel sets of $I(e)$ are compact, so $e_n$ has a subsequential limit $e^* \in \ce \cap O^c$. Thus, $e^* \geq e$ and $e^* \neq e$. By Corollary \ref{cor: monotone}, we have that $I(e^*) > I(e)$. At the same time, by the lower semi-continuity of $I$, we have that and $s =\lim I(e_n) \geq I(e^*)$.

\section{Large Deviations of Geodesics} \label{sec: geo}
In this section, we shall provide a proof of Theorem \ref{thm: geodesics}. Fix $u =(p;q) = (x,s;y,t) \in \oR$ and let $\gamma_n$ denote the almost surely unique geodesic of $T_n$ with endpoints $u$. For a path $\gamma$ in $e$ we can define the weight function $w :[s,t] \to \bR^{\ge 0}$ as 
\[
 	w(r) = \abs{\gamma_{[s,r]}}_e \mathrm{for}\; r \in [s,t].
\]

We say that $p \:\|\: q$ if $(p,q)=(x,s;y,t) \in \oR$ and $\abs{x-y} = t-s$. Let $\cG$ denote the set of all curves in $\oR$ with endpoints $u$ and notate
\[
	\cG_b := \bc{f \in C([s,t]) : \exists r \in (s,t) \text{ such that } p \:\|\: (f(r),r) \text{ or } (f(r),r) \:\|\: q.}
\]
It would be nice to think of Theorem \ref{thm: geodesics} as an application of the contraction principle to Theorem \ref{thm: ldp}. The next lemma tells us that if finite rate metrics hypo-converge to another finite rate metric then their geodesics might not converge to a geodesic of the limiting metric. Thus, making it slightly trickier to make continuity based arguments. However, this shall not be a big deterrent in using the idea of contraction principle.

\begin{lemma} \label{lemma: gdtech}
	Let $e_n$ be a convergent sequence in $\cD$, converging to $e \in \cD$ and let $f_n$ be $e_n$-geodesics with endpoints $u$ such that $f_n$ converges uniformly to $f$. Then one of the following must happen
	\begin{enumerate}
		\item $f$ is an $e$-geodesic with endpoints $u$.
		\item All $e$-geodesics with endpoints $u$ are in $\cG_b$.
	\end{enumerate} 
	Furthermore, $\exists e' \in \cD_u(f)$ such that $e' \leq e$.
\end{lemma}
\begin{proof}
	If $\abs{x-y} =t-s$ then there is only one path in $\oR$ with endpoints $u$ and the claim follows from the existence of geodesics (Lemma \ref{lemma: gdf}). So assume otherwise and fix a finite partition $\cP$ of $[s,t]$. Hypo-convergence of $e_n$ implies that
	\[
		\limsup_n \abs{f_n}_{e_n} \leq \limsup_n \abs{f_n}_{e_n, \cP} \leq \abs{f}_{e, \cP}.
	\]
	We can now take an infimum over all such partitions in the R.H.S. which gives us that $\limsup_n \abs{f_n}_{e_n} \leq \abs{f}_{e}$. Since, $f_n$ are $e_n-$geodesics, we have that
	\[
		\limsup_n e_n(u) \leq \abs{f}_{e}.
	\]
	By our assumptions on $u$, we have that for $\epsilon>0$ small enough the point $u_\epsilon=(x, s+\epsilon; y, t-\epsilon) \in \oR$ and by Remark \ref{rmk: hypo-bound} we have that
	\[
		e(u_\epsilon) \leq \abs{f}_{e}.
	\]
	Letting $\epsilon \to 0$, we can conclude that $\abs{f}_e \geq \abs{g}_e$ for all $g \in \cG \setminus \cG_b$ proving the first part of the claim. 

	Since $e \in \cD$, it must have a complete network say $\Gamma$. The graph of function $f$ defines a geodesic curve in $\oR$, call it $\gamma_0$. We can intersect the network $\Gamma$ with $\gamma_0$, and construct a new network $\Gamma' =\bc{\gamma \cap \gamma_0\: |\: \gamma \in \Gamma}$. By Proposition \ref{prop: e-structure} we can construct a finite rate metric $e'$ using the network $\Gamma'$ and $e$ such that $e' \leq e$ and $I(e') \leq I(e)$. Moreover, by construction $e' \in \cD_u(f)$.
\end{proof}
\begin{corollary} \label{cor: gdappli}
	The infimum in \eqref{def: Ju} is always achieved whenever it is finite. Moreover, any metric achieving the infimum is given by plating the path $f$ with some weight function $w$. Thus for $f \in \cG$, we can realize $\cJ_u(f)$ as the value of the optimization problem:
	\begin{align*}
		\text{minimize} && &\cR(f, w) :=  \sup_{\cP \in \cC_{[s,t]}} \sum_{[s_{i}, t_i] \in \cP} J_{m_i} \br{\frac{w(t_i) -w(s_i)}{t_i -s_i}}(t_i-s_i), \\
		\text{subject to} && &w(b) - w(a) \geq d((f(a), a), (f(b), b)), \;\;\text{for all }s \leq a \leq b \leq t,
	\end{align*}
	where $m_i = (f(t_i)-f(s_i))/(t_i-s_i)$, and we borrow other notations from Section \ref{sec: ntwrk}.
\end{corollary}
\begin{proof}
	Let $f$ be a curve in $\oR$ with endpoints $u$ such that $\cJ_u(f)$ is finite. There exists a sequence $e_n \in \cD_u(f)$ such that $I(e_n) \downarrow \cJ_u(f)$. Thus, $I(e_n)$ is bounded and by Lemma \ref{lemma: Dm cpt} there is a convergent subsequence of $e_n$, say with limit point $e$. By lower semi-continuity of $I$, we have that
	$I(e) \leq \liminf_n I(e_n) =\cJ_u(f)$. By Lemma \ref{lemma: gdtech}, we have a metric $e' \in \cD_u(f)$ such that $I(e') \leq I(e)$. Thus, the infimum is achieved at $e'$. The variational follows by calculating the rate of such a minimizing metric.
\end{proof}

\begin{remark} \label{rmk: sec9}
	Fix endpoints $u \in \oR$ and a weight function $w:[s,t] \to \bR^{\ge 0}$. Then $\cR(f, w)$ is minimized by the straight line joining the endpoints. This follows simply from the convexity of the point-to-point rate function $J$. If $J$ is strictly convex in the parameter $\gamma$, then the minimizer is unique. Thus, on the event that the last passage time has a large positive deviation, the geodesic is localized around the straight line connecting the endpoints. Such a localization result was shown for Poisson LPP in \cite[Theorem 3]{DZ99} by explicit computation of the corresponding $J$ function. To see that the result holds for Exponential/Geometric LPP, one can simply take the derivative of $J$ (which is explicitly known) with respect to $\gamma$ and see that it is not a constant function for any fixed $x$. As we know that $J$ is convex, this means that it must be strictly convex in $\gamma$. 
\end{remark}

We now use the above ideas to prove a couple of technical lemmas that will be useful for proving the large deviation lower bound.

\begin{lemma} \label{lemma: uniqgd}
	Let $f \in \cG$ such that $\cJ_u(f) < \alpha < \infty$, then there exists $e \in \cD$ with $I(e) < \alpha$, such that $f$ is the unique $e-$geodesic with endpoints $u$.
\end{lemma}
\begin{proof}
	Since $\cJ_u(f) < \infty$, we can find a planted network metric $e_f$ such that $I(e_f) = \cJ_u(f)$. Furthermore, this metric is generated by planting $f$ with an appropriate weight function $w$.	Let $\kappa >0$, and consider the weight function $w_\kappa(z) = w(z) +\kappa(z-s)$ for $z \in[s,t]$. We can now construct a new metric $e$ by planting the path $f$ with weight function $w_\kappa$, using Proposition \ref{prop: e-structure}. The metric $e$ must also have finite rate for some $\kappa$ small enough as $\Theta(e_f)$ would be unbounded otherwise.
	
	Convexity of the point to point rate function implies that $\cR$ is convex. Therefore,
	\[
		\cR(f, (1- \lambda) w + \lambda w_\kappa) \leq (1- \lambda) \cR(f,  w) + \lambda \cR(f, w_\kappa) = (1- \lambda) I(e_f) + \lambda I(e). 
	\]
	Thus, for $\lambda$ small enough, we have that $\cR(f, w_{\kappa \lambda}) = \cR(f, (1- \lambda) w + \lambda w_\kappa) < \alpha$. Thus, we can choose $\kappa$ small enough so that $I(e) < \alpha$.

	Since $f$ was an $e_f-$geodesic, $f$ must be the unique $e-$geodesic with endpoints $u$.
\end{proof}

\begin{lemma} \label{lemma: perturb}
	Let $f \in \cG_b$ such that $\cJ_u(f) < \alpha < \infty$ and $\abs{x-y} \neq t-s$, then for all $\epsilon>0$ there exists $f' \in \cG \setminus \cG_b$ with $\cJ_u(f') < \alpha$ and $\norm{f-f'}_u < \epsilon$. 
\end{lemma}
\begin{proof}
	Let $\ell$ denote the function which defines the straight line between the endpoints $u$. For $\lambda \in (0,1)$, we define $f_\lambda := \lambda \ell + (1-\lambda) f$. We can pick $\lambda$ such that the ${\norm{f_\lambda - f}}_u$ is arbitrary small and by construction, $f_\lambda \notin \cG_b$. Furthermore for any interval $[a,b] \subseteq [s,t]$,
	\[
		\abs{\frac{f_\lambda(b) -f_\lambda(a)}{b-a} - \frac{f(b) -f(a)}{b-a}} \leq 2 \lambda. 
	\] 
	By the sublinearity and continuity of the shape function $F$ (Property \ref{A1}), we can claim that $F(\gamma t, t)/t$ is uniformly continuous in $\gamma$. Thus for any interval $[a,b] \subseteq [s,t]$,
	\begin{equation} \label{eq: perturb1}
		d((f_\lambda(a), a), (f_\lambda(b), b)) \leq d((f(a), a), (f(b), b)) + \delta(b-a),
	\end{equation}
	for some $\delta$ depending only on $\lambda$ which goes to zero with $\lambda$.

	Let $w$ be the weight function satisfying the variational problem in Corollary \ref{cor: gdappli} corresponding to $f$. By \eqref{eq: perturb1}, the weight function $w_\delta(z) := w(z) + \delta(z-s),$ for $z \in [s,t]$ will satisfy the constraint in the variational problem corresponding to $f_\lambda$. Convexity of $\cR$ and Remark \ref{rmk: sec9} implies that
	\begin{align*}
		\cR(f_\lambda, w_\delta) &\leq (1-\lambda) \cR(f, w_\delta) + \lambda \cR(\ell, w_\delta)  \leq \cR(f, w_\delta).
	\end{align*}
	We can now bound the right-hand side by $\alpha$ as in the proof of Lemma \ref{lemma: uniqgd} by choosing $\lambda$ to be sufficiently small.
\end{proof}

We now present a proof for Theorem \ref{thm: geodesics}.

\begin{proof}[Proof of Theorem \ref{thm: geodesics}]
	
	\underline{\textbf{Lower semi-continuity:}} Consider $f_n \to f$ uniformly and suppose that limit infimum of $\cJ_u(f_n)$ is finite, otherwise there is nothing to prove. Following Corollary \ref{cor: gdappli} we can find $e_n \in \cD$ achieving the infimum \eqref{def: Ju} for each $f_n$. Since, $I$ is a good rate function, there is a subsequential limit $e$ of $e_n$, and $I(e) \leq \liminf_n I(e_n) = \liminf_n \cJ_u(f_n)$. By Lemma \ref{lemma: gdtech}, $\exists e' \in \cD_u(f)$ such that $I(e') \leq I(e)$. Thus,
	\[
		\cJ_u(f) \leq I(e') \leq \liminf_n \cJ_u(f_n).
	\]

	\underline{\textbf{Compact sublevel sets:}} Lower semi-continuity of the rate function implies that the sublevel sets are closed. Notice that all functions in $C([s,t])$ with finite rate must define a curve in $\oR$. This forces the functions to be Lipschitz continuous with Lipschitz constant one. At the same time these functions are uniformly bounded with the bound only depending on the endpoints $u$. Compactness now follows from the Arzela-Ascoli theorem. 

	In particular, we have shown that $\cG$ is pre-compact.

	\underline{\textbf{Large deviation upper bound:}} Let $A$ be a closed subset of $C([s,t])$, and let $S_A \subset \cD$ be the set of metrics that have a geodesic with endpoints $u$ which lies in $A$. Then, by Theorem \ref{thm: ldp}
	\begin{equation} \label{eq: sec9i}
		\P{\gamma_n \in A} = \P{T_n \in S_A} \leq \exp\br{(o(1) - \inf_{\bar S_A} I)n}.
	\end{equation} 
	The upper bound follows if for every $e \in \bar S_A$ with $I(e) \leq \alpha < \infty$, we can find $f \in A$ with $\cJ_u(f) \leq \alpha$. There exists a sequence $(e_n) \subset S_A$ such that $e_n \to e$. Let $f_n \in A$ be $e_n$-geodesics with endpoints $u$, since $\cG$ is precompact $f_n$ must have a subsequential limit $f \in A$. By Lemma \ref{lemma: gdtech} we can construct $e' \in \cD_u(f)$ such that $I(e') \leq I(e)$. Hence, 
	\[
		\P{T_n \in S_A} \leq \exp\br{(o(1) - \inf_{S_A } I)n} = \exp\br{(o(1) - \inf_{A} \cJ_u)n}. 
	\]

	\underline{\textbf{Large deviation lower bound:}} Let $A \subseteq C([s,t])$ be open and assume that $\cJ_u(f) < \infty$ for some $f \in A$. Following \eqref{eq: sec9i} and Theorem \ref{thm: ldp}, we have that 
	\[
		\liminf_n \frac{1}{n} \log \P{\gamma_n \in A} \geq -\inf_{S_A^\circ} I,
	\]
	where we can take the interior relative to $\cD$, as $I = \infty$ outside it. Hence, the desired lower bound follows if for every $f \in A$ with $\cJ_u(f) < \alpha < \infty$, we can find $e \in S_A^\circ$ with $I(e) \leq \alpha$. We divide the argument into two cases, for $f \in \cG_b$ and otherwise.
	
	Let $f \notin \cG_b$, by Lemma \ref{lemma: uniqgd} there exists  $e \in \cD$ with $I(e) < \alpha$ and $f$ as the unique geodesic across the endpoints $u$. We claim that $e \in S_A^\circ$. Let $e_n$ be a sequence in $\cD$ converging to $e$ and let $g_n$ be $e_n-$geodesics with endpoints $u$. Since $g_n$ define curves in $\oR$, they must have a subsequential limit $g$. By Lemma \ref{lemma: gdtech}, $g$ must be an $e-$geodesic with endpoints $u$. Uniqueness forces $f=g$ and thus $g_n \in A$ for infinitely many $n$ since $A$ is open. This forces $e_n \in S_A$ for infinitely many $n$ and thus $e$ must be in  $S_A^\circ$.
	
	If $u$ is such that $\abs{x-y} = t-s$, then the above arguments hold for all $f \in \cG$. Otherwise for $f \in \cG_b$, we can apply Lemma \ref{lemma: perturb} to reduce the problem to the previous case which completes our proof.	
\end{proof}
We can now apply Theorem \ref{thm: geodesics} to prove Corollary \ref{cor: corner}. The argument is mostly computational, hence we provide a sketch of the proof and leave it to the readers to fill in the details. All notations are borrowed from the statement of Corollary \ref{cor: corner}. 

\begin{proof}[Sketch of Proof of Corollary \ref{cor: corner}]
	Let $A \subset C([0,1])$ denote the set of all functions which passes through the point $a$. Since $\cJ_u$ is a good rate function, it must attain a minimizer $f$ with weight function $w$ in $\bar A$. The journey from $t =0$ to $t=1/2$ is identical to the journey from $t=1/2$ to $t=1$, hence $f$ and $w$ can  be taken to be symmetric about $t= 1/2$. Let $\ell$ denote the function corresponding to joining the endpoint to $a$ by straight lines. Let $T = w(1) -w(0)$ and $w_\ell$ denote the weight configuration which distributes $T$ uniformly on $[0,1]$. Convexity of the point-to-point rate function tells us that the cost of planting $\ell$ with $w_\ell$ is less than that of planting it with $w$. Combining this with Remark \ref{rmk: sec9}, we observe that taking $\ell$ with weight configuration $w_\ell$ will have less cost than planting $f$ with weights $w$. Furthermore, $\ell$ is a geodesic with weights $w_\ell$ if $T \geq d(u)$ and the cost will be minimized at $T= d(u)$. This tells us that
	$$\inf_{\bar A} \cJ_u = \cJ_u(\ell) = J(m, d(u)),$$
	where $m$ is the slope of the line connecting the origin and $a$. This directly gives us the required upper bound. For the lower bound, it is enough to note that the set of all functions which pass through $a+(0, \epsilon)$ is in the interior of $A$ for all $\epsilon$ small enough. Thus, 
	\[
		\inf_{A^\circ} \cJ_u \leq \inf_\epsilon J(m + 2 \epsilon, d(u)) = J(m, d(u)), 
	\]
	where the last inequality follows by the continuity of $J$. Thus, we have the matching upper and lower bounds. 
\end{proof}
\section{Discussion on Properties} \label{sec: properties}

In this section we shall show that all the properties that we required from the LPP models in Section \ref{s: intro} can be proven for a large class of models. In particular, we aim to provide proofs and references for the properties under the assumption that the random weights $w_{ij} \sim \mu$ are non-negative, unbounded and satisfy \eqref{ass: mom}. 

\begin{proposition}[{\cite[Theorem 2.3]{Mart04}}] \label{prop: A1}
	For LPP models with i.i.d. $\mu$ weights satisfying \eqref{ass: mom}, there exists a concave continuous sublinear function $F: \vee \to [0,\infty)$, such that for $t \in \bR^{\geq 0}, \abs{x} \leq t$, 
	\[
		\lim_{n \to \infty} \frac{T\br{(0,0), (\floor{nx},\floor{nt})}}{n} = F(x,t),
	\]
	uniformly on compact sets, almost surely. Moreover, the function $F$ is strictly positive for $t >0$. 
\end{proposition}
We remark that the above proposition holds for much weaker moment assumptions. In particular, for last passage percolation on $\bZ^d$, \cite{Mart04} shows that Property \ref{A1} holds if
\begin{equation} \label{ass: weak mom}
	\E{w^d (\log w)^{1+\epsilon}} < \infty,
\end{equation}
for some $\epsilon >0$.

The next Proposition will tell us that Property \ref{A2} holds under the assumption \ref{ass: mom}. Before we state the result, we recall that under the Assumption \ref{ass: mom}, one can define the Cramer's rate function $I_c$ corresponding to $\mu$ as the Legendre transform of the logarithmic moment generating function. Furthermore, the assumption that $\mu$ has unbounded support ensures that $I_c$ is finite on $\bR$.

\begin{proposition}[{\cite[Theorem 3.2]{GS13}}, {\cite[Theorem 5.2]{Kes86}}] \label{prop: A2}
	Consider an LPP models with i.i.d. $\mu$ weights satisfying \eqref{ass: mom} and support of $\mu$ is unbounded. For $\abs{\gamma} \leq 1$, there exists a function $J$ such that
	\[
		\lim_{n \to \infty} \frac{1}{n} \log \P{T_n \br{(0,0), (\gamma , 1)} \geq  x} = -J(\gamma, x).
	\] 
	Furthermore, the function $J$ can be computed as
	\begin{equation} \label{eq: J}
		-J(\gamma, x) = \sup_n \frac{1}{n} \log \P{T_n \br{(0,0), ({\gamma }, 1)} \geq x},
	\end{equation}
	and $J$ satisfies the following properties:
	\begin{enumerate}[(i)]
		\item The function $J$ is jointly convex and continuous.
		\item $J(\gamma, x) =0$ for $x \leq F(\gamma, 1)$.
		\item The function $J_\gamma = J(\gamma, \cdot)$ is strictly increasing for $x \geq F(\gamma, 1)$.
	\end{enumerate}
\end{proposition}
\begin{proof}
	For $\abs{\gamma} \leq 1$, the existence and convexity of the function $J(\gamma, \cdot)$ follow from the superadditivity of the sequence 
	\[
		\br{\log \P{T \br{(0,0), (\floor{\gamma n}, n)} \geq n x}}_n.
	\]
	The superadditivity also implies that we can compute $J$ as
	\[
		-J(\gamma, x) = \sup_n \frac{1}{n} \log \P{T \br{(0,0), (\floor{\gamma n}, n)} \geq nx}.
	\]
	We claim that the function $J$ is also jointly convex in $\gamma$. Let $(\gamma, x) = \alpha (\gamma_1, x_1) + (1- \alpha) (\gamma_2, x_2)$, where $\abs{\gamma_1}, \abs{\gamma_2} \leq 1$. We have that
	\begin{align*}
		\P{T \br{(0,0), (\floor{\gamma n}, n)} \geq n x} &\geq	\P{T \br{(0,0), (\floor{\gamma_1 \alpha n}, \alpha n)} \geq \alpha n  x_1} \times \\
		&\P{T \br{(0,0), (\floor{\gamma_2 (1-\alpha)n}, (1-\alpha)n)} \geq (1-\alpha)n x_2}.
	\end{align*}
	Taking the log and letting $n$ tend to infinity will prove the desired convexity claim. The convexity immediately tells us that $J$ is an upper semi-continuous function, and it is continuous on $(-1,1) \times \bR$ when it is finite. Notice that along the directions $\abs{\gamma} =1$, the rate function is simply given by the Cramer's rate function $I_c$ corresponding to $\mu$. Convexity of $J$ forces $J(\gamma, \cdot) \leq J(1, \cdot) = I_c(\cdot) < \infty$. Thus, $J$ is continuous on $(-1,1) \times \bR$.  
	
	Last passage times can be realized as a limit of the point-to-point partition function for directed polymers in the same random environment. The rate functions for these two quantities are related by the following inequalities,
	\begin{align*}
		J(\gamma, x-c\beta^-1) \leq J^\beta(\gamma, \beta x) \leq J(\gamma, x).
	\end{align*} 
	In the above equation $c$ is an absolute constant, $\beta$ is the temperature parameter for the directed polymer and $J^\beta$ represents the rate function of the partition function with origin and $n(\gamma, 1)$ as endpoints. The function $J^\beta$ is known to be continuous under the assumption \ref{ass: mom} (see \cite[Theorem 3.2]{GS13}). We now let $(\gamma_n, x_n) \to (\gamma,x)$ and $\beta \to \infty$ in the above inequality to get that
	\[
		J(\gamma, x) \leq \liminf_{n} J(\gamma_n, x_n).
	\]  
	The strong law of large number in Proposition \ref{prop: A1} implies that $J(\gamma, x) =0$ for $x < F(\gamma, 1)$ and continuity of $J$ allows us to extend it to $x \leq F(\gamma, 1)$. The claim about strictly increasing follows from \cite[Theorem 5.2]{Kes86}.
\end{proof}
We end this session with a proof of Property \ref{A3} for a general class of LPP models. It was remarked in \cite{BGS19} that one can use their results in conjunction with the proof of \cite[Theorem 5.2]{Kes86} to establish Property \ref{A3}. As the result was not stated formally, we shall instead provide an elementary proof for the property by combining their ideas.

\begin{proposition} \label{prop: A3}
	Consider an LPP model with i.i.d. weights that satisfies Property \ref{A1}. Fix $\abs{\gamma} <1$ and $x>0$, then for $n$ large enough
	\[
		\P{T_n((0,0), (\gamma , 1)) \leq F(\gamma, 1) - x} \leq \exp \br{-c_xn^2},
	\]
	for some constant $c_x>0$ depending on $x$.
\end{proposition}
\begin{proof}
	We consider the rectangle $R$ with origin and $(\gamma n, n)$ as opposite vertices. Let $M \in \bN$ large which we shall fix later and $K =n/M$. Divide $R$ into $K^2$ identical pieces by lines of slope 1 and -1 as shown in Figure \ref{fig1}. Let us label the intersection points of these lines as
	\[
		v_{i,j} = \br{\frac{n}{2K}\br{i(1+\gamma) -j(1-\gamma)}, \frac{n}{2K}\br{i(1+\gamma) +j(1-\gamma)}}.
	\]
	
	\begin{figure} 
		\includegraphics{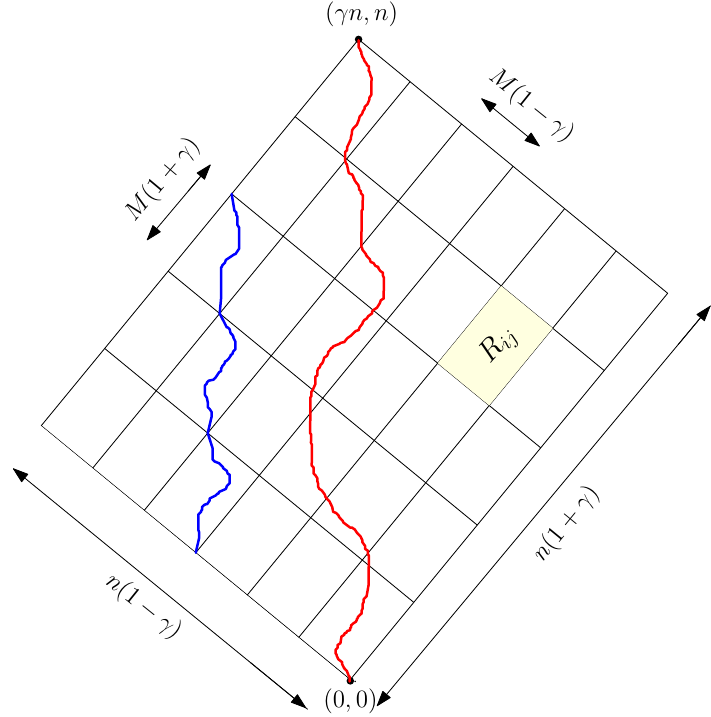}
		\caption{Illustration for the construction of Proposition \ref{prop: A3}. The blue line denotes the path corresponding to $X_k$ and the red line illustrates the geodesic.}
		\label{fig1}
	\end{figure}
	We enumerate the smaller rectangles by $\bc{R_{ij}}_{i,j=1}^K$, where $R_{ij}$ denotes the rectangle with opposite vertices $v_{i-1, j-1}$ and $v_{i,j}$. For notational convenience, we shall assume that the vertices of every rectangle $R_{ij}$ have integer coordinates. It will be clear from the proof that our arguments can be easily adapted to the case when this is not true.
	
	We denote by $X_{ij}$, the last passage time between $v_{i-1, j-1} + (0,1)$ and $v_{i,j} -(0,1)$ and let 
	\[\cA := \bc{T_n((0,0), (\gamma , 1)) \leq F(\gamma, 1) - x}.\]
	On the event $\cA$, for $\abs{k} \in \bc{0, \cdots, K-1}$, we have that
	\[
		\sum_{i-j = k} X_{i,j} \leq T((0,0), (\gamma n , n)) \leq n\br{F(\gamma, 1) - x}.
	\]
	Let us denote the sum in the L.H.S. by $X_k$. The collection of random variables $\bc{X_k}$ is independent. Thus,
	\begin{equation} \label{eq: A3con}
		\P{\cA} \leq \prod_{\abs{k} \leq K-1} \P{X_k \leq n\br{F(\gamma, 1) - x}},
	\end{equation}
	and we denote the event in the R.H.S. by $\cA_k$. Fix indices $i,j$ and let
	\[
		\tilde{X}_{ij} = \frac{-1}{M-2} \min(X_{ij} - \E{X_{ij}} , 0).
	\]
	These random variables are bounded above by $(M-2)^{-1}\E{X_{ij}}$. By Property \ref{A1} we have an $M_0$ such that for $M \geq M_0$,
	\begin{equation} \label{eq: A3i}
		\frac{1}{(M-2)} \E{X_{ij}} \in (F(\gamma, 1) - x/4, F(\gamma, 1)).
	\end{equation}
	Combined with the almost sure convergence of $X_{ij}$, we can assert that for $M$ large enough
	\[\E{\tilde X_{ij}} \leq x/4.\]
	As $X_{ij}$ are identically distributed, the above inequality holds for all indices $i$ and $j$. $X_k$ is the sum of $K -\abs{k}$ many copies of $X_{11}$ and,
	\begin{align*}
		\E{X_k} &\geq (M-2)(K- \abs{k}) (F(\gamma, 1) - x/4)\\
		& \geq n (1- 2/M)(1- \abs{k}/K) (F(\gamma, 1) - x/4) \\
		& \geq n \br{F(\gamma, 1) - x/2}.
	\end{align*} 
	Where the last inequality holds for
	$$\abs{k} \leq k_0 = K \br{\frac{x}{4F(\gamma, 1)} - 2/M}.$$
	Therefore, we can bound the probability of $\cA_k$ as
	\begin{align*}
		\P{\cA_k} &\leq \P{X_k \leq \E{X_k}- nx/2} \\
		&\leq \P{\sum_{i-j=k} (X_{ij} -\E{X_{ij}}) \leq  -nx/2}  \leq \P{\frac{1}{K} \sum_{i-j=k} \tilde{X}_{ij} \geq x/2} \leq \exp \br{- c K x^2},
	\end{align*}		
	where we have used the Azuma Hoeffding's inequality in the last step. Thus,
	\begin{equation*}
		\P{\cA} \leq \exp \br{- c K k_0x^2} \leq \exp\br{-cM^{-2} x^3 n^2},
	\end{equation*}
	for some absolute constant $c >0$. 
\end{proof}

\begin{remark}
	We have stated the proposition for last passage percolation on $\bZ^2$. However, the same arguments can be modified for higher dimensions and will give us a decay of $\exp(-c_x n^d)$ in $\bZ^d$.
\end{remark}

\section{Poisson Last Passage and Polymer Models} \label{sec: pois}
The abstract treatment for the theory of last passage models makes it much easier to extend the result to similar models. In particular, we will be reviewing the changes required to fit our framework to Poisson LPP and directed polymer models.

\subsection{Poisson Last Passage}
Consider a 2-dimensional Poisson point process on the plane $\bR^2$ with rate $\lambda =1$. Define the weight of an upright path as the number of Poisson points on it. The last passage value from $u$ to $v$ is defined as the maximal weight of an upright path with endpoints $u$ and $v$. If the points $u,v$ are such that there is no upright path between them then we set the last passage value to zero. The definition is identical to that in the lattice models. The only difference being that we no longer have to use floor/ ceiling functions as the model is defined on a continuous space. 

One can easily check that this definition satisfies the generalized triangle inequality \eqref{eq: triangle}. Furthermore, all the three properties that we required for lattice models also hold in this case. Property \ref{A1} follows form the shear invariance in the model, and Kingman's subadditivity theorem. The precise function follows from \cite{AD95} using the connection with the longest increasing subsequence in random permutations. Property \ref{A2} and Property \ref{A3} are known to be true following \cite{seppalainen1998large} and \cite{DZ99}.

Thus, all the tools we used for the large deviation principle (Theorem \ref{thm: ldp}) are available in this setting, and we can claim that Poisson LPP also satisfies the LDP. The large deviation principle for geodesics also follow from the same arguments.

\subsection{Directed Polymers}
The random environment for these models is exactly the same as what we have in the LPP case. We will be looking at the discrete lattice $\bZ^2$ and assign to each vertex an i.i.d.\ non-negative weight. We define the weight of a path $\pi$ (denoted by $\ell(\pi)$) as the sum over all the vertices on it (both endpoints included). Instead of looking over the maximal weight of an upright path, we now ask about the total weight of all upright paths between two fixed vertices. Thus, we define the partition function as
\[
	\cZ(u,v) = \sum_{pi: u \to v} \ell(\pi), 
\] 
where the sum is over all upright paths from $u$ to $v$. We define the free energy as the log of the partition function. These functions can be extended to $\bR^2$ by similar use of floor/ceiling functions as in the case of LPP times. We also do a similar rotation of the standard axes to align them with space and time coordinates. 

Fix $(p,q) = (x,s;y,t) \in \oR$, and let $r \in [s,t)$, definition of the partition function implies that
\begin{align} \label{eq: dp_triangle}
	\cZ(p,q) &\geq \max_{p \leq (z,r) \leq q} \cZ(x,s;z,r) \cZ((z,r)^+; y,t)  \notag,\\  
	\cZ(p,q) &\leq \abs{t-s} \max_{p \leq (z,r) \leq q} \cZ(x,s;z,r) \cZ((z,r)^+; y,t).
\end{align}	
We shall be interested in the scaled free energy,
\begin{equation*}
\log \cZ_n(u,v) := \frac{1}{n} \log \cZ(nu,nv).
\end{equation*}

It follows that scaled the free energy satisfies the generalized triangle inequality \eqref{eq: triangle}. Property \ref{A1} and Property \ref{A2} follow from \cite[Theorem 3.2]{GS13} under the same assumptions on vertex weights, in fact their results are what allowed us to prove Property \ref{A2} for the LPP case. Finally, Property \ref{A3} can be proven almost identically to the LPP. Thus, the basic tools that are required for our theory still hold under the same assumptions on the random weights.

One can now attempt to use the theory from Section \ref{sec: frm} and \ref{sec: ntwrk} to prove the large deviation principle. There is just one little problem, free energies do not satisfy the metric composition law exactly. Thus, in the proof of the upper bound, one cannot directly apply Corollary \ref{corollary: mcl}. However, we can show that free energies satisfy an approximate version of the metric composition law and use this property instead. 

We say that a sequence $e_n \in \cE$ satisfy an approximate metric composition law if 
\begin{align}\label{eq: app_mcl}
	e_n(p,q) =  \max_{p \leq (z,r) \leq q} e_n(x,s;z,r) + e_n((z,r)^+; y,t) + \epsilon_n,
\end{align}
for a sequence $\epsilon_n$ converging to zero. Following \eqref{eq: dp_triangle}, the scaled free energy satisfies
\begin{align*} 
	\log \cZ_n(p,q) \geq \max_{p \leq (z,r) \leq q} \log\cZ_n(x,s;z,r) + \log \cZ_n((z,r)^+; y,t),
\end{align*}
as well as 
\begin{align}
	\log \cZ_n(p,q) \leq  \max_{p \leq (z,r) \leq q} \log\cZ_n(x,s;z,r) +\log\cZ_n((z,r)^+; y,t) +\frac{1}{n} \log (n\abs{t-s}).
\end{align}
Combining the above equations shows that the scaled free energy satisfy an approximate metric composition law. We now show that a sequence satisfying \eqref{eq: app_mcl} cannot converge to a metric without the metric composition law. The proof is very similar to the proof of Lemma \ref{lemma: Dm cpt}.

\begin{proposition} \label{prop: app_mcl}
	Let $e_n \in \cE$ be a sequence of metrics that satisfy \eqref{eq: app_mcl}, $e \in \cE$ be a metric which does not satisfy Definition \ref{def: D}(iii). Then there exists $\epsilon >0$ such that for $n$ large enough $\bd(e_n, e) >\epsilon$. 
\end{proposition}
\begin{proof}
	Assume otherwise, then we have that $e_n \to e$ in the $\tau_h$ topology. Since $e$ does not satisfy the metric composition law, we can find $(p,q) = (x,s;y,t) \in \oR$ and $r \in (s,t)$ such that,
	\begin{equation}
		e(p,q) \geq \sup_{\bc{z: p < (z,r) < q}} e(p, (z,r)) + e((z,r)^+, q),
	\end{equation}
	for $ (p,q) \in \oR$, and $r \in (s,t)$. 
	
	We can furthermore assume that $(p,q)$ and $r$ are such that they satisfy the hypothesis of Lemma \ref{lemma: max}. If no such $p,q$ and $r$ exist then Lemma \ref{lemma: max} would imply that $e$ satisfies the metric composition law. Let $I_r := \bc{z: p < (z,r) < q}$ as in the proof of Lemma \ref{lemma: max}. Let $p_n = (x, s_n)$ and $q_n = (y, t_n)$ where $(s_n), (t_n) \subset R_e$, $s_n \uparrow s$ and $t_n \downarrow t$. As in the proof of Lemma \ref{Arz}, can choose the sequence in such a way so that
	\[
		e(p,q) = \lim_n e_n(p_n, q_n).
	\]
	By Lemma \ref{lemma: max}, we can pick $z_n$ such that
	\[
		e_n(p_n, q_n) = e_n(p_n; (z_n, r)) + e_n((z_n,r)^+; q_n) + 2\epsilon_n.
	\]
	If needed, we can pass through a subsequence so that $z_n$ is convergent and let $z_0$ be the limit. Note that $z_0$ must belong to the interval $I_r$. Hence,
	\begin{align*}
		e(p, q) &= \limsup_n e_n(p_n, q_n) \\
		&= \limsup_n e_n(p_n, (z_n,r)) + e_n((z_n,r)^+, q_n) + 2\epsilon_n \\
		&\leq \limsup_n e_n(p_n, (z_n,r)) + e_n((z_n,r), q_n) + 2\epsilon_n\\
		&\leq e(p, (z_0,r)) + e((z_0,r), q) \qquad (\Rightarrow \Leftarrow ). 
	\end{align*}
	Thus, $e$ must satisfy the metric composition law. This completes the proof of our claim.
\end{proof}

One can now prove the large deviation upper bound for $\log \mathcal{Z}_n$ by the exact same arguments as in Section \ref{sec: pf-ub}. The only difference is in the proof of Lemma \ref{lemma: Dm bnd}, one know argues using Proposition \ref{prop: app_mcl}, instead of Corollary \ref{corollary: mcl}. The proof of the lower bound remains unchanged and thus the same large deviation theory works for directed polymer models. 

\subsection*{Acknowledgement}
I am grateful to B\'alint Vir\'ag for his invaluable guidance throughout this project. I also thank Riddhipratim Basu for his insightful comments on related work.
\bibliography{bib}
\bibliographystyle{plain}
\end{document}